\documentclass[3p,times]{elsarticle}
\graphicspath{{./pics/}} 
\usepackage{ecrc}

\volume{}

\firstpage{1}

\runauth{}

\usepackage{amssymb}
\usepackage{amsthm}

\usepackage{hyperref}       
\usepackage{algorithm}
\usepackage{algpseudocode}
\usepackage{bbm}
\usepackage{amsmath}
\usepackage{subfigure}
\usepackage{xcolor}

\DeclareMathOperator*{\minimize}{minimize}

\DeclareMathOperator{\tr}{tr}

\newtheorem{thm}{Theorem}[section]

\newtheorem{lem}[thm]{Lemma}
\newtheorem{rem}[thm]{Remark}
\newtheorem{prop}[thm]{Proposition}

\newcommand{\grad}{\mbox{grad\,}}

\newcommand{\D}{\mbox{D\,}}

\newcommand{\ci}{\mathbbm{i}}
\newcommand{\dt}[1]{\left.{\frac{d}{dt}}\right|_{t={#1}}}

\newcommand{\bmat}[1]{\begin{bmatrix} #1 \end{bmatrix}}

\newcommand{\norm}[1]{\left\lVert#1\right\rVert}
\newcommand{\abs}[1]{\left \lvert #1 \right \rvert}
\newcommand{\ip}[1]{\left\langle #1 \right\rangle}
\newcommand{\MIN}[2]{\begin{array}{ll} \displaystyle \minimize_{#1} & {#2} \end{array}}

\newcommand{\StatexIndent}[1][3]{%
  \setlength\@tempdima{\algorithmicindent}%
  \Statex\hskip\dimexpr#1\@tempdima\relax}

\def\Box{\mbox{ }\rule[0pt]{1.5ex}{1.5ex}}

\allowdisplaybreaks

\usepackage[figuresright]{rotating}

\begin{document}

\begin{frontmatter}



\dochead{}

\title{On the convergence of orthogonalization-free conjugate gradient method for extreme eigenvalues of Hermitian matrices: a Riemannian optimization interpretation}

\author[shixin]{Shixin Zheng}
\author[haizhao]{Haizhao Yang}
\author[xiangxiong]{Xiangxiong Zhang}
\address[shixin]{Department of Mathematics, Purdue University, West Lafayette, IN, USA, \texttt{zheng513@purdue.edu}}
\address[haizhao]{Department of Mathematics, University of Maryland College Park, MD, USA, \texttt{hzyang@umd.edu}}
\address[xiangxiong]{Department of Mathematics, Purdue University, West Lafayette, IN, USA, \texttt{zhan1966@purdue.edu}}

\begin{abstract}
In many applications, it is desired to obtain extreme eigenvalues and eigenvectors of large Hermitian matrices by efficient and compact algorithms. In particular, orthogonalization-free methods are preferred for large-scale problems for finding eigenspaces of extreme eigenvalues without explicitly computing orthogonal vectors in each iteration. For the top $p$ eigenvalues, the simplest orthogonalization-free method is to find the best rank-$p$ approximation to a positive semi-definite Hermitian matrix by algorithms solving the unconstrained Burer-Monteiro formulation.   
We show that the nonlinear conjugate gradient method for the unconstrained Burer-Monteiro formulation is equivalent to a Riemannian conjugate gradient method on a quotient manifold with \textcolor{black}{the Bures-Wasserstein} metric, thus its global convergence to a stationary point can be proven. Numerical tests suggest that it is efficient for computing the largest $k$ eigenvalues for large-scale matrices if the largest $k$ eigenvalues are nearly distributed uniformly. 
\end{abstract}


\end{frontmatter}


\section{Introduction}
\label{sec1}
\setcounter{equation}{0}
\setcounter{figure}{0}
\setcounter{table}{0}

\subsection{The eigenvalue problem of Hermitian positive definite matrices}
In this paper, we are interested in solving the eigenvalue problem for a  Hermitian matrix $B \in \mathbb{C}^{n\times n}$ to find its largest $p$ eigenvalues and the corresponding eigenvectors.
For large enough $\mu>0$, $A:=B+\mu I\in \mathbb{C}^{n\times n}$ is a positive definite Hermitian matrix with the same extreme eigenspaces. Thus we focus only on Hermitian positive definite or semi-definite matrices.

Extreme eigenvalue problems for Hermitian matrices naturally arise in many applications \cite{shi2000normalized, cheeger2015lower, donath1972algorithms, fiedler1973algebraic,lu2017preconditioning,Lu2016ACS,wang2019coordinate}. For example, many problems can be cast as a graph, for which the adjacency matrix and the graph Laplacian are real symmetric thus Hermitian \cite{Pang2022ADB}.    The extreme eigenvalues and eigenvectors of these matrices contain information about the graph and the point cloud data such as diffusion maps \cite{coifman2005geometric}. Notice that the discussion in this paper also applies to the smallest $k$ eigenvalues for a positive definite Hermitian matrix $B$ by considering either $A=\mu I-B$ with large enough $\mu$ or $A=B^{-1}$  if an efficient implementation of linear system solver for $Bx=b$ is available, i.e., the matrix-vector multiplication $B^{-1}b$ can be efficiently implemented.   

The extreme eigenvalue problem can be written as an optimization problem, with many different cost functions to consider. The most well-known one is to minimize the multicolumn Rayleigh quotient
\begin{equation}
    \MIN{x \in \mathbb{C}^{n\times p}}{f(x) := \tr\left((x^*x)^{-1}x^*Ax \right).}
\end{equation}
 If assuming the spectrum of $x^*x$ is bounded by one and take the inverse of $x^*x$ as the first order approximation of the Neumann series expansion, then as an approximation to multicolumn Rayleigh quotient, a popular method known as orbital minimization method (OMM) is to minimize the cost function \cite{Corsetti2013TheOM}: 
 \begin{equation}\label{min_OMM}
    \MIN{x \in \mathbb{C}^{n\times p}}{f(x) :=  \tr\left((2I - x^*x)x^*Ax\right)}.
\end{equation}
Another simple formulation is to consider optimization over the noncompact Stiefel manifold $\mathbb{C}^{n\times p}_*=\{X\in \mathbb{C}^{n\times p}$: \mbox{rank}(X)=p\}: 
\begin{equation}\label{min_prob_BM_NS_manifold}
    \MIN{x \in \mathbb{C}^{n\times p}_*}{f(x) := \frac12\norm{xx^* - A}_F^2},
\end{equation}
where $\|\cdot\|_F$ is the matrix Frobenius norm. 
Various orthogonalization-free algorithms for solving both \eqref{min_OMM}
and \eqref{min_prob_BM_NS_manifold} were considered and compared numerically in \cite{gao2022triangularized}.

 {A third choice is LOBPCG method first introduced in  \cite{doi:10.1137/S1064827500366124}. A critical step in the LOBPCG method is a Rayleigh-Ritz procedure in which an orthonormal basis is computed to simplify calculations and ensure numerical stability and it is the only orthogonalization step. LOBPCG without orthogonalization also gives an orthogonalization-free method, which   may still work well for many problems in practice, though it
might suffer from some instability when the number of eigenpairs to be computed becomes large. Careful base selection strategies \cite{HETMANIUK2006324} \cite{doi:10.1137/17M1129830} can improve its robustness.}

\subsection{The real inner product and Fr\'{e}chet derivatives}
 In this paper, we mainly focus on the cost function \eqref{min_prob_BM_NS_manifold} and consider the nonlinear conjugate gradient (CG) methods solving \eqref{min_prob_BM_NS_manifold}.
 
Since $f(x)$ is real-valued and thus not holomorphic,  $f(x)$ does not have a complex derivative with respect to $x\in\mathbb C^{n\times p}$. 
The linear spaces of complex matrices will therefore be regarded as vector spaces over $\mathbb R$. 
For any real vector space $\mathcal{E}$, the inner product on $\mathcal{E}$ is denoted by $\ip{.,.}_\mathcal{E}$.  For real matrices $A,B \in \mathbb{R}^{n\times p}$, the Hilbert–Schmidt inner product is $\ip{A,B}_{\mathbb{R}^{n \times p}} = {\tr(A^T B)}$.
Let  $\Re({A})$ and $\Im(B)$ represent the real and imaginary parts of a complex matrix $A$.
For $A,B \in \mathbb{C}^{n\times p}$, the real inner product for the real vector space $\mathbb{C}^{n\times p}$ then equals
 \begin{equation}
     \ip{A,B}_{\mathbb{C}^{n\times p}} := \Re({\tr(A^* B)}),
     \label{real-inner-product}
 \end{equation} where ${}^*$ is the conjugate transpose. We emphasize that \eqref{real-inner-product} is a real inner product, rather than the complex Hilbert–-Schmidt inner product. It is straightforward to verify that \eqref{real-inner-product} can be written as
 \[\ip{A,B}_{\mathbb{C}^{m\times n}}=\tr(\Re(A)^T \Re(B))+\tr(\Im(A)^T \Im(B))=\ip{\Re(A),\Re(B)}_{\mathbb{R}^{m\times n}}+\ip{\Im(A),\Im(B)}_{\mathbb{R}^{m\times n}}. \]
With the real inner product \eqref{real-inner-product} for the real vector space $\mathbb C^{n\times p}$, 
a Fr\'{e}chet derivative for the real-valued function $f(x)$ can be defined as
\begin{equation}
     \nabla{f}(x) = \nabla{f}_{\Re(x)}(x)+\ci\nabla{f}_{\Im(x)}(x)
     \in \mathbb{C}^{n\times p},
     \label{frechet-derivative-1}
\end{equation} 
where $\nabla{f}_{\Re(x)}(x), \nabla{f}_{\Im(x)}(x) \in \mathbb R^{n\times p}$
are the gradient of the cost function $f$ with respect to the real and imaginary parts of $x$, respectively. 
In particular, for $f(x)=\frac12\|\mathcal A(xx^*)-b\|_F^2$ with a linear operator $\mathcal A$, the Fr\'{e}chet derivative \eqref{frechet-derivative-1} becomes
\[ \nabla{f}(x)=2\mathcal A^*(\mathcal A(xx^*)-b)x,\]
where $\mathcal A^*$ is the adjoint operator of $\mathcal A$. 
See Appendix in \cite{zheng2023riemannian}  for details. 

\subsection{The conjugate gradient method solving the Burer-Monteiro formulation}

Notice that $\mathbb{C}^{n\times p}_*$ is an open set in the Euclidean space $\mathbb{C}^{n\times p}$, thus any line search method $x_{k+1}=x_k+\alpha_k \eta_k$ starting with the iterate $x_k\in \mathbb{C}^{n\times p}_*$ and a small enough step size $\alpha_k$ will give $x_{k+1}\in \mathbb{C}^{n\times p}_*$. Therefore, any such line search algorithm can be regarded as the same algorithm solving an unconstrained problem with a non-degenerate $x_k\in \mathbb{C}^{n\times p}_*$:
\begin{equation}\label{min_prob_BM}
    \MIN{x \in \mathbb{C}^{n\times p}}{f(x) := \frac12\norm{xx^* - A}_F^2}.
\end{equation}

In the literature, the formulation \eqref{min_prob_BM} is often called the Burer-Monteiro method for Hermitian positive semi-definite (PSD) fixed rank $p$ constraint, i.e., for minimizing $\|X-A\|_F^2$ where $X$ is a Hermitian PSD matrix of rank $p$. 

The nonlinear conjugate gradient method for \eqref{min_prob_BM} can be written as
\begin{equation}
    \begin{cases}
x_{k+1}&=x_k+\alpha_k \eta_k,\\
\eta_{k+1}&=-\nabla f(x_k)+\beta_k\eta_k=-2(xx^*-A)x+\beta_k\eta_k,
\end{cases}
\label{BMCG}
\end{equation}
where $\alpha_k$ is the step size, $\beta_k$ is a nonlinear coefficient computed by various formulae, \textcolor{black}{and $\eta_k$ is the search direction in CG method.}
 In this paper, we only consider two variants for how to compute $\beta_k$: one is the Polak–Ribi\'{e}re CG method, and the other one is the Fletcher-Reeves CG method for computing the conjugate direction
\cite{nocedal1999numerical}. 

\subsection{The main result: the convergence of Riemannian conjugate gradient method via quotient geometry}
 
 The CG method \eqref{BMCG} for finding top $p$ eigenvalues of Hermitian PSD matrix $A$ has been considered in \cite{gao2022triangularized}. In particular, \eqref{BMCG} does not require any orthogonalization operation in each iteration, and its performance is superior especially for uniformly distributed eigenvalues in numerical tests. 
 
 The landscape of \eqref{min_prob_BM} has been well studied in \cite{gao2022triangularized, liu2015efficient, jin2017escape,  li2019coordinatewise} and its local minimizers must also be 
global minimizers. 
 Theorem 2.1 in \cite{gao2022triangularized} implies that, if $\hat x\in \mathbb C^{n\times p}_*$
satisfies $\nabla f(\hat x)=0$ for $f(x)=\frac12\|xx^*-A\|_F^2$, then $\hat x=UO$ where  $O\in \mathbb C^{p\times p}$ is a unitary matrix, and $U\in \mathbb C^{n\times p}$ has orthogonal columns as some eigenvectors of $A$. 
Furthermore, any local minimum is a global minimum, i.e., any local minimizer of \eqref{min_prob_BM} in $\mathbb C^{n\times p}_*$ has the form $\hat x=UO$ with columns of $U$ being eigenvectors of a Hermitian PSD matrix $A$ corresponding to its top $p$ eigenvectors. 

However, the convergence of CG method \eqref{BMCG} for \eqref{min_prob_BM} has never been rigorously justified. 

Notice that there is an ambiguity up to unitary matrices in both formulations \eqref{min_prob_BM} and \eqref{min_prob_BM_NS_manifold}, that is $f(xO)=f(x)$ for any $O\in\mathcal O_p$, where $\mathcal O_p$ are all $p\times p$ unitary matrices. To this end, mathematically it is proper to consider an equivalence class for each $x\in \mathbb C_*^{n\times p}$:
\[[x] = \{xO: \forall O\in \mathcal{O}_p \},\]
and a quotient set 
\[\mathbb{C}^{n\times p}_*/\mathcal{O}_p:=\{[x]: \forall x\in \mathbb C_*^{n\times p}\}.\]

The quotient set with a proper metric becomes a quotient manifold. It is not uncommon to abuse notation by letting $x$ denote the equivalent class $[x]$, and $\overline x$ denote one representation of this equivalent class.
So we can instead consider the optimization over the quotient manifold:
\begin{equation}\label{min_prob_quotient_manifold}
	\MIN{x \in \mathbb{C}^{n\times p}_*/\mathcal{O}_p }{h(x):= f(\overline x) = \frac{1}{2}\norm{\overline x\overline x^*-A}^2_F }.
\end{equation} 

Following the recent progress in \cite{zheng2023riemannian} for Riemannian optimization over Hermitian PSD fixed rank manifolds, we first show that the simple unconstrained Burer-Monteiro CG method \eqref{BMCG} is equivalent to a Riemannian CG method solving \eqref{min_prob_quotient_manifold} over the quotient manifold $\mathbb{C}^{n\times p}_*/\mathcal{O}_p$ with \textcolor{black}{the Bures-Wasserstein metric \cite{massart2019curvature}} and proper retraction and vector transport operators. Then with existing Riemannian optimization convergence theory, we can establish the global convergence of the simple algorithm \eqref{BMCG} to a stationary point of \eqref{min_prob_BM_NS_manifold}.   \textcolor{black}{We emphasize that the main result of this paper is the global convergence proof for the classical simple algorithm \eqref{BMCG}, and we do not modify the algorithm  \eqref{BMCG} at all. The Riemannian optimization is used only for proving convergence of  \eqref{BMCG}, and \eqref{BMCG} should not be implemented via much more complicated Riemannian optimization over a quotient manifold.}

\subsection{Related work and contributions}

To be more specific, we will show that both the Polak–Ribi\'{e}re CG method and the Fletcher-Reeves CG method in \eqref{BMCG} are equivalent to their Riemannian variants over the quotient manifold $\mathbb{C}^{n\times p}_*/\mathcal{O}_p$ with \textcolor{black}{the Bures-Wasserstein metric \cite{massart2019curvature}}.

Moreover, this equivalence allows us to establish the global convergence of the conventional Fletcher-Reeves CG method \eqref{BMCG} to a stationary point of \eqref{min_prob_BM_NS_manifold}, following the convergence of the Riemannian Fletcher-Reeves CG method in \cite{sato2015new}. For the problem  \eqref{min_prob_BM}, it has been well known that local minima are also global minima \cite{liu2015efficient, jin2017escape,  li2019coordinatewise, gao2022triangularized}, e.g., critical points are either global minima or saddle points.  
Combined with the result that first-order methods almost always avoid strict saddle points \cite{lee2019first}, we obtain a justification of the global convergence of the conventional Fletcher-Reeves CG method \eqref{BMCG} to the global minimizer of \eqref{min_prob_BM_NS_manifold}.
For the Polak–Ribi\'{e}re CG method, the convergence is much harder to establish, but its numerical performance is often superior. 

In the literature, notable convergence results for orthogonalization-free methods include global convergence of perturbed gradient descent for \eqref{min_prob_BM}  in \cite{jin2017escape} and global convergence of TriOFM in \cite{gao2021global}.

The same CG algorithm \eqref{BMCG} was also considered in \cite{gao2022triangularized} for real symmetric matrices. Both our algorithm and convergence proof also apply to the Hermitian matrices. 
We also verify the numerical performance of the discussed algorithms on large matrices of the size millions by millions. In particular, our numerical tests for large matrices are consistent with the observation in \cite{gao2022triangularized} that the simple CG method \eqref{BMCG} is superior for nearly uniformly distributed extreme eigenvalues. 

\textcolor{black}{This paper mainly focuses on the convergence analysis of the simplest orthogonalization-free method \eqref{BMCG} which is fully scalable in parallel computing. Developing distributed and parallel numerical implementation will be left as future work. In the literature, most numerical solvers for eigenvalue problems rely on orthogonalization to achieve high efficiency in sequential computing. Well-developed algorithms with orthogonalization include \cite{doi:10.1137/S1064827500366124,ZHOU2006172,NEYMEYR2006114,COAKLEY2013379}. To achieve better parallel efficiency for a full eigendecomposition, spectrum slicing can be applied to estimate different eigenpairs in different spectrum regions simultaneously \cite{AKTULGA2014195,doi:10.1137/15M1054493,PhysRevB.79.115112,10.14492/hokmj/1272848031,doi:10.1137/16M1061965,doi:10.1137/16M1086601}. }

\subsection{Outline of this paper}

We first review basic concepts and known results for Riemannian quotient manifolds $\mathbb{C}_*^{n\times p}/\mathcal{O}_p$ in Section \ref{sec-notation}. 
Then we review the equivalence of the conventional CG method to the Riemannian CG method in Section \ref{sec-cg}. The convergence proof of the Riemannian CG method is provided in Section \ref{sec-proof}. 
In Section \ref{sec-CRGD},  \textcolor{black}{we show that the simple coordinate descent method of minimizing \eqref{min_prob_BM} is also equivalent to a coordinate Riemannian gradient descent method. }
Section \ref{sec-tests} includes numerical tests. 
Concluding remarks are given in Section \ref{sec-remarks}.

\section{Preliminaries: Riemannian Quotient Manifold $\mathbb{C}_*^{n\times p}/\mathcal{O}_p$}
\label{sec-notation}

In this section, we briefly review some known results of the Riemannian geometry of $\mathbb{C}_*^{n\times p}/\mathcal{O}_p$  that will be used in this paper. Any missing details can be found in \cite{zheng2023riemannian}. 

\subsection{$\mathbb{C}_*^{n\times p}/\mathcal{O}_p$ as a quotient manifold}

Define
$\mathbb{C}^{n\times p}_*=\{X\in \mathbb{C}^{n\times p}: \mbox{rank}(X)=p\}$
and an equivalence relation on  $\mathbb{C}^{n\times p}_*$ through the smooth Lie group action of unitary matrices $\mathcal{O}_p$ on the manifold $\mathbb{C}^{n\times p}_*$:
\[
\begin{aligned}
     \mathbb{C}^{n\times p}_* \times \mathcal{O}_p &\rightarrow& \mathbb{C}^{n\times p}_*, \qquad 
    (\overline{x},O) &\mapsto& \overline{x}O.
\end{aligned}
\]
This action defines an equivalence relation on $\mathbb{C}^{n\times p}_*$ by setting $\overline{x}_1 \sim \overline{x}_2$ if there exists an $O \in \mathcal{O}_p$ such that $\overline{x}_1 = \overline{x}_2 O$. Hence we have constructed a quotient space $\mathbb{C}^{n\times p}_*/\mathcal{O}_p$ that removes this ambiguity. The set $\mathbb{C}^{n\times p}_*$ is called the \textit{total space} of $\mathbb{C}^{n\times p}_*/\mathcal{O}_p$. 

Denote the natural projection as
\[
\begin{aligned}
\pi: \mathbb{C}^{n\times p}_* &\rightarrow& &\mathbb{C}^{n\times p}_*/\mathcal{O}_p, \qquad 
\overline{x} &\mapsto& & x. 
\end{aligned}
\]
We denote the equivalence class containing $\overline{x}$ as 
$$[\overline{x}]=\pi^{-1}(x) =\left\{ \overline{x}O \vert O\in \mathcal{O}_p \right\}. $$

Following Corollary 21.6 and Theorem 21.10 of \cite{lee_introduction_2012}, $\mathbb{C}^{n\times p}_*/\mathcal{O}_p$ is a smooth manifold as stated in the following theorem.
 
\begin{thm}
The quotient space $\mathbb{C}^{n\times p}_*/\mathcal{O}_p$ is a quotient manifold over $\mathbb R$ of dimension  $2np-p^2$ and has a unique smooth structure such that the natural projection $\pi$ is a smooth submersion. 
 
\end{thm}

\subsection{Vertical space}
The equivalence class $[\overline{x}] = \pi^{-1}(x)$ is an embedded submanifold of  $\mathbb{C}^{n\times p}_*$(\cite[Prop.~3.4.4]{absil_optimization_2008}). The tangent space of $[\overline{x}] $ at $\overline{x}$ is therefore a subspace of $\mathbb{C}^{n\times p}$ called the \textit{vertical space} at $\overline{x}$ and is denoted by $\mathcal{V}_{\overline{x}} $. The following proposition characterizes $\mathcal{V}_{\overline{x}}$. 
 
 \begin{prop}
The vertical space at $\overline{x}\in [\overline{x}]=\left\{ \overline{x}O \vert O\in \mathcal{O}_p \right\} $, which is the tangent space of $[\overline{x}] $ at $\overline{x}$ is 
\[
    \mathcal{V}_{\overline{x}} = \left\{\overline{x} \Omega\vert \Omega^*  = - \Omega, \Omega \in \mathbb{C}^{p\times p} \right\}. 
\]
\end{prop}

\subsection{Riemannian metric} 
 A \textit{Riemannian metric} $g$ is a smoothly varying inner product defined on the tangent space. That is, $g_{\overline{x}}(\cdot,\cdot)$ is an inner product on $T_{\overline{x}}\mathbb{C}^{n\times p}_*$. Once we choose a Riemannian metric $g$ for $\mathbb{C}^{n\times p}_*$, we can obtain the orthogonal complement in $T_{\overline{x}} \mathbb{C}^{n\times p}_*$ of $\mathcal{V}_{\overline{x}}$ with respect to the metric. In other words, we choose the \textit{horizontal distribution} as orthogonal complement w.r.t. Riemannian metric, see \cite[Section 3.5.8]{absil_optimization_2008}. This orthogonal complement to $\mathcal{V}_{\overline{x}}$ is called \textit{horizontal space} at $\overline{x}$ and is denoted by $\mathcal{H}_{\overline{x}}$. We thus have
\begin{equation}
\label{tangentspace-decomp}
    T_{\overline{x}}\mathbb{C}^{n\times p}_*= \mathcal{H}_{\overline{x}} \oplus \mathcal{V}_{\overline{x}}.
\end{equation}

Once we have the horizontal space, there exists a unique vector $\overline{\xi}_{\overline{x}} \in \mathcal{H}_{\overline{x}}$ that satisfies $\D\pi(\overline{x})[\overline{\xi}_{\overline{x}}] = \xi_{x}$ for each $\xi_{x} \in T_{x}\mathbb{C}^{n\times p}_*/\mathcal{O}_p$. This  $\overline{\xi}_{\overline{x}}$ is called the \textit{horizontal lift} of $\xi_{x}$ at $\overline{x}$.

In this paper, we consider the Riemannian metric on $\mathbb{C}^{n\times p}_*$ to be the canonical Euclidean inner product on $\mathbb{C}^{n\times p}$ defined by 
\begin{equation}\label{eqn:riemannian_metric}
    g_{\overline{x}}(A,B) := \ip{A,B}_{\mathbb{C}^{n\times p}} = \Re(\tr(A^*B)), \quad \forall A,B\in T_{\overline{x}}\mathbb{C}^{n\times p}_* = \mathbb{C}^{n\times p}. 
\end{equation}

\begin{prop}
Under metric $g$ defined in \eqref{eqn:riemannian_metric}, the horizontal space at $\overline{x}$ satisfies 
\begin{eqnarray*}
    \mathcal{H}_{\overline{x}} &=& \left\{z\in \mathbb{C}^{n\times p}: \overline{x}^*z =z^*\overline{x} \right\} = \left\{\overline{x}(\overline{x}^*\overline{x})^{-1}S +\overline{x}_\perp K \vert S^* = S, S \in \mathbb{C}^{p\times p}, K \in \mathbb{C}^{(n-p)\times p} \right\}.
\end{eqnarray*}
\end{prop}

\subsection{Projections onto vertical space and horizontal space}\label{section:quotient_projection}
Due to the direct sum property
\eqref{tangentspace-decomp}, for our choices of $\mathcal{H}_Y$, there exist projection operators for any $z \in T_Y\mathbb{C}^{n\times p}_* = \mathbb{C}^{n\times p}$ to $\mathcal{H}_Y$ as
\begin{equation*}
    z = P_{\overline{x}}^{\mathcal{V}}(z) + P_{\overline{x}}^{\mathcal{H}}(A).
\end{equation*}

It is straightforward to verify the following formulae for projection operators $P_Y^\mathcal{V}$ and $P_Y^{\mathcal{H}}$.

\begin{prop}\label{prop:projections_quotient}
The orthogonal projections  of any $z \in \mathbb{C}^{n \times p}$ to $\mathcal{V}_{\overline{x}}$ and $\mathcal{H}_{\overline{x}}$ are 
\begin{equation*}
    P^\mathcal{V}_{\overline{x}}(z) = \overline{x}\Omega,\quad  P^{\mathcal{H}}_{\overline{x}} (z)= z - \overline{x}\Omega,
\end{equation*}
where $\Omega$ is the skew-symmetric matrix that solves the Lyapunov equation
\[
    \Omega \overline{x}^*\overline{x} + \overline{x}^*\overline{x} \Omega = \overline{x}^* z- z^* \overline{x}.
\]
\end{prop}
\begin{rem}
\label{rem-Lyapunov}
The solution $X$ to the Lyapunov equation $XE +EX = Z$ for a Hermitian $E$ is unique if $E$ is Hermitian positive-definite \cite[Section 2.2]{massart_quotient_2020}. Let $E = U\Lambda U^*$ be the SVD, then the Lyapunov equation  $XE +EX = Z$ becomes 
\begin{equation*}
     (U^*XU)\Lambda +\Lambda (U^*XU)  = U^*ZU, 
\end{equation*}
which gives the solution 
\begin{equation*}
    (U^*XU)_{i,j} = (U^*ZU)_{i,j}/(\Lambda_{i,i}+\Lambda_{j,j}).
\end{equation*}
\end{rem}

\subsection{$\mathbb{C}^{n\times p}_*/\mathcal{O}_p$ as  Riemannian quotient manifold}
First, we show in the following lemma the relationship between the horizontal lifts of the quotient tangent vector $\xi_{x}$ lifted at different representatives in $[\overline{x}]$. 
\begin{lem}\label{lem:horizontal_lift}
Let $\eta$ be a vector field on $\mathbb{C}^{n\times p}_*/\mathcal{O}_p$, and let $\overline{\eta}$ be the horizontal lift of $\eta$. Then for each $\overline{x}\in \mathbb{C}^{n\times p}_*$, we have $$\overline{\eta}_{\overline{x}O} = \overline{\eta}_{\overline{x}}O,\qquad \forall O\in \mathcal{O}_p. $$  
\end{lem}
\begin{proof}
See \cite[Prop.~A.8]{massart_quotient_2020} 
\end{proof}

Recall from \cite[Section 3.6.2]{absil_optimization_2008} that if the expression $g_{\overline{x}}(\overline{\xi}_{\overline{x}},\overline{\zeta}_{\overline{x}})$ does not depend on the choice of $\overline{x}\in \pi^{-1}(\overline{x})$ for every $x\in \mathbb{C}^{n\times p}_*/\mathcal{O}_p$ and every $\xi_{x}, \zeta_{x}\in T_{x}\mathbb{C}^{n\times p}_*/\mathcal{O}_p$, then 
\begin{equation}\label{eqn:riemannian_metric_quotient}
    g_{x}(\xi_{x},\zeta_{x}) := g_{\overline{x}}\left(\overline{\xi}_{\overline{x}},\overline{\zeta}_{\overline{x}}\right)
\end{equation}
defines a Riemannian metric on the quotient manifold $\mathbb{C}^{n\times p}_*/\mathcal{O}_p$. By Lemma \ref{lem:horizontal_lift}, it is straightforward to verify that the Riemannian metric \eqref{eqn:riemannian_metric} on $\mathbb{C}^{n\times p}_*$ induces a Riemannian metric on $\mathbb{C}^{n\times p}_*/\mathcal{O}_p$ defined as \eqref{eqn:riemannian_metric_quotient}. The quotient manifold $\mathbb{C}^{n\times p}_*/\mathcal{O}_p$ endowed with a Riemannian metric defined in (\ref{eqn:riemannian_metric_quotient}) is called a \textit{Riemannian quotient manifold}. By abuse of notation, we use $g$ for denoting Riemannian metrics on both total space $\mathbb{C}^{n\times p}_*$ and quotient space $\mathbb{C}^{n\times p}_*/\mathcal{O}_p$.
\textcolor{black}{This particular metric is also call the Bures-Wasserstein metric  for PSD matrices of fixed-rank \cite{massart2019curvature}.}

\subsection{Riemannian gradient}
The cost function of \eqref{min_prob_BM} induces a cost function on $\mathbb{C}^{n\times p}_*/\mathcal{O}_p$.
\begin{equation}\label{eqn:cost_function_quotient}
\begin{aligned}
    h: \mathbb{C}^{n\times p}_*/\mathcal{O}_p & \rightarrow \mathbb{C}, \qquad 
    x & \mapsto f(\overline{x}). 
\end{aligned}
\end{equation}

That is, $f = h \circ \pi$. Notice when we solve \eqref{min_prob_BM_NS_manifold}, we restrict $f$ on the noncompact Stiefel manifold $\mathbb{C}^{n\times p}_*$, which is a submanifold of $\mathbb{C}^{n\times p}$. Hence the Riemannian gradient of $f$ on $\mathbb{C}^{n\times p}_*$ at $\overline{x}$ is the projection of the Fr\'{e}chet gradient of $f$ on $\mathbb{C}^{n\times p}$, denoted by $\nabla f(\overline{x})$,   onto the tangent space $T_{\overline{x}} \mathbb{C}^{n \times p}_* = \mathbb{C}^{n \times p}$. Since $\nabla f $ is already in $\mathbb{C}^{n\times p}$, the projection is identity. That is,
\begin{equation}\label{eqn:equivalence_grad}
\grad f(\overline{x}) = \nabla f(\overline{x}).     
\end{equation}

\begin{rem}
One can refer to \cite[Appendix A]{zheng2023riemannian} for more details about Fr\'{e}chet derivative. A Fr\'{e}chet gradient for any real-valued function $f(X)$ at $X \in\mathbb{C}^{m\times n}$ can be defined as
\begin{equation}
     \nabla{f}(X) = \nabla{f}_{\Re(X)}(X)+\ci\nabla{f}_{\Im(X)}(X)
     \in \mathbb{C}^{m\times n},
     \label{frechet-derivative-2}
\end{equation} 
where $\nabla{f}_{\Re(X)}(X), \nabla{f}_{\Im(X)}(X) \in \mathbb R^{m\times n}$
are the gradient of  $f$ with respect to the real and imaginary parts of $X$, respectively. 
In particular, for the cost function considered in this paper $f(\overline{x})=\frac12 \|\overline{x}\overline{x}^*-A\|_F^2$, the Fr\'{e}chet gradient \eqref{frechet-derivative-2} becomes
\[ \nabla{f}(\overline{x})=2(\overline{x}\overline{x}^* -A)\overline{x}.\]
\end{rem}

Now consider the Riemannian gradient of $h$ at $x \in \mathbb{C}^{n\times p}_*/\mathcal{O}_p$. $\grad h(x)$ is a tangent vector in $T_{x}\mathbb{C}^{n\times p}_*/\mathcal{O}_p$ . The next theorem shows that the horizontal lift of $\grad h(x)$ can be obtained from the Riemannian gradient of $f$.

\begin{thm}\label{thm:lift_gradient_quotient} The horizontal lift of the Riemannian gradient of $h$ at $\overline{x}$ is the Riemannian gradient of $f$ at $\overline{x}$. That is, 
\[
\overline{\grad h(x) }_{\overline{x}}  = \grad f(\overline{x}).
\]
Therefore, although $\grad f(\overline{x})$ belongs in $\mathbb{C}^{n \times p}$, it is automatically in $\mathcal{H}_{\overline{x}}$. 
\end{thm}
\begin{proof}
See \cite[Section 3.6.2]{absil_optimization_2008}. 
\end{proof}

\subsection{Retraction}
\label{sec:retraction}
The retraction on the quotient manifold $\mathbb{C}^{n\times p}_*/\mathcal{O}_p$ can be defined using the retraction on the total space $\mathbb{C}^{n\times p}_*$.  
Let $Y\in \mathbb{C}^{n\times p}_*$, for any $Z \in \mathbb{C}^{n\times p}$ and a step size $\tau>0$,
\[
\overline{R}_Y(\tau Z):= Y + \tau Z, 
\]
is a retraction on $\mathbb{C}^{n\times p}_*$ if  $Y+\tau Z$ remains full rank, which is ensured for small enough $\tau$. 
Then Lemma \ref{lem:horizontal_lift} indicates that $\overline{R}$ satisfies the conditions of \cite[Prop.~4.1.3]{absil_optimization_2008}, which implies that 
\begin{equation}\label{eqn:retraction_quotient}
    R_{x}(\tau \eta_{x}) := \pi(\overline{R}_{\overline{x}} (\tau  \overline{\eta}_{\overline{x}})) = \pi(\overline{x}+\tau \overline{\eta}_{\overline{x}})
\end{equation}
defines a retraction on the quotient manifold $\mathbb{C}^{n\times p}_*/\mathcal{O}_p$
for a small enough step size $\tau >0.$

\subsection{Vector transport}
We use differentiated retraction as our vector transport \cite[Section 8.1.4]{absil_optimization_2008}.
\begin{equation}\label{eqn:vector_transport}
    \mathcal{T}_{\eta_x}(\xi_x) := \D R_{x}(\eta_x)[\xi_x] = \dt{0}{R_{x}(\eta_x + t\xi_x)}. 
\end{equation}
Notice that 
\begin{eqnarray*}
        \mathcal{T}_{\eta_x}(\xi_x) &=& \D R_{x}(\eta_x)[\xi_x]= \D \pi\left(\overline{R}_{\overline{x}}(\overline{\eta}_{\overline{x}})\right)[\D \overline{R}_{\overline{x}}(\overline{\eta}_{\overline{x}})[\overline{\xi}_{\overline{x}}] ] = \D \pi\left(\overline{x} + \overline{\eta}_{\overline{x}}\right)\left[ \dt{0}{\overline{R}_{\overline{x}}}\left(\overline{\eta}_{\overline{x}} + t\overline{\xi}_{\overline{x}}\right) \right] \\
        &=& \D \pi\left(\overline{x} + \overline{\eta}_{\overline{x}}\right)\left[ \dt{0}{\left(\overline{x} +  \overline{\eta}_{\overline{x}} + t \overline{\xi}_{\overline{x}}  \right)}\right] = \D \pi\left(\overline{x} + \overline{\eta}_{\overline{x}}\right)\left[  \overline{\xi}_{\overline{x}} \right]= \D \pi\left(\overline{x} + \overline{\eta}_{\overline{x}}\right)\left[  P^\mathcal{H}_{\overline{x} + \overline{\eta}_{\overline{x}}}\left(  \overline{\xi}_{\overline{x}} \right) \right]. 
    \end{eqnarray*}
    
Hence the horizontal lift of a transported vector is simply the projection of the original horizontal lift to the new horizontal space, as shown in the following formula.
\begin{equation}\label{eqn:vector_transport_lift}
\overline{\mathcal{T}_{\eta_x}(\xi_x) }_{\overline{x}+\overline{\eta}_{\overline{x}}} = P^\mathcal{H}_{\overline{x}+\overline{\eta}_{\overline{x}}}(\overline{\xi}_{\overline{x}}).
\end{equation}

\section{The Conjugate Gradient Methods}
 \label{sec-cg}

We first recall the traditional conjugate gradient method for solving \eqref{min_prob_BM}, which is summarized as Algorithm \ref{alg:CG}. We present the abstract Riemannian conjugate gradient method for solving \eqref{min_prob_quotient_manifold} over the quotient manifold as Algorithm \ref{alg:RCG}, with Wolfe conditions
\begin{equation}\label{wolfe_1}
    h(R_{x_k}(\alpha_k \eta_k)) \leq h(x_k) +c_1 \alpha_k g_{x_k}(\grad h(x_k), \eta_k),
\end{equation}
\begin{equation}\label{wolfe_2}
    \abs{g_{R_{x_k}(\alpha_k \eta_k)}(\grad h(R_{x_k}(\alpha_k \eta_k)), \D R_{x_k}(\alpha_k \eta_k) [\eta_k])}\leq c_2 \abs{g_{x_k}(\grad h(x_k), \eta_k)}. 
\end{equation}
$0< c_1 < c_2 < 1$.

The abstract Algorithm \ref{alg:RCG} can be implemented as Algorithm \ref{alg:RCG_implement}, in which each tangent vector is treated as horizontal lift and each iterate is a representative of its equivalence class, and it is independent of the choice of the representative of the equivalent class.

\begin{algorithm}[htbp]
\caption{(Polak–Ribi\'{e}re or   Fletcher-Reeves) Conjugate Gradient on $\mathbb{C}^{n\times p}$}\label{alg:CG}
\begin{algorithmic}[1]
\Require initial iterate ${x}_0 \in \mathbb{C}^{n\times p}$, tolerance $\varepsilon>0$, initial descent direction as negative gradient $\eta_0 = -\nabla f(x_0) = -2(x_0x_0^* - A)x_0$
\For{ $k =0,1,2,\dots$}
    \State{Use backtracking to compute the step size $\alpha_k >0$ satisfying the strong Wolfe conditions}
    \State{Obtain the new iterate by }
    \[
    x_{k+1} = x_k + \alpha_k \eta_k
    \]
    \State{Compute the gradient}
    \par\hskip\algorithmicindent ${\xi}_{k+1} :={\nabla f(x_{k+1})}$
    \State{Check for convergence}
     \par\hskip\algorithmicindent if $\norm{{\xi}_{k+1}}_F < \varepsilon $, then break 
    \State Compute a conjugate direction by the Polak–Ribi\'{e}re method or the  Fletcher-Reeves method
    \par\hskip\algorithmicindent ${\eta}_{k+1}= -{\xi}_{k+1} + \beta_{k+1} \eta_{k}$
    \[  \text{ where }  \beta_{k+1}           = \left\{
    \begin{aligned}
    & \max\left(0, \frac{\ip{\nabla f({x}_{k+1}), \nabla f({x}_{k+1})-\nabla f(x_k)}}{\ip{ \nabla f({x}_{k}), \nabla f({x}_{k}) }}\right) & \quad \mbox{if using Polak–Ribi\'{e}re} \\ 
    & \frac{\ip{\nabla f({x}_{k+1}), \nabla f({x}_{k+1})}}{\ip{ \nabla f({x}_{k}), \nabla f({x}_{k})}} &  \quad \mbox{if using Fletcher-Reeves.}
    \end{aligned}
    \right.
    \]
\EndFor
\end{algorithmic}
\end{algorithm}

\begin{algorithm}[htbp]
\caption{Riemannian Conjugate Gradient on the quotient manifold $\mathbb{C}^{n\times p}_*/\mathcal{O}_p$ with metric $g$}\label{alg:RCG}
\begin{algorithmic}[1]
\Require initial iterate ${x}_0 \in \mathbb{C}^{n\times p}_*/\mathcal{O}_p$, tolerance $\varepsilon>0$, tangent vector $\eta_0 = -\grad h(x_0)$
\For{ $k =0,1,2,\dots$}
    \State{Compute the step size $\alpha_k >0$ satisfying the strong Wolfe conditions \eqref{wolfe_1} and \eqref{wolfe_2}}
    \State{Obtain the new iterate by retraction}
    \[
    x_{k+1} = R_{x_k}(\alpha_k \eta_k)
    \]
    \State{Compute the gradient}
    \par\hskip\algorithmicindent ${\xi}_{k+1} :={\grad h(x_{k+1})}$ 
    \State{Check for convergence}
     \par\hskip\algorithmicindent if $\norm{{\xi}_{k+1}}:=\sqrt{g_{{x}_{k+1}}({\xi}_{k+1}, {\xi}_{k+1})} < \varepsilon $, then break 
    \State Compute a conjugate direction by the Polak–Ribi\'{e}re (PR$_+$) method or the  Fletcher-Reeves (FR) method, and vector transport
    \par\hskip\algorithmicindent ${\eta}_{k+1} = -{\xi}_{k+1} + \beta_{k+1} \mathcal{T}_{\alpha_{k} \eta_{k}}(\eta_{k})$
    \[  \text{ where }  \beta_{k+1} = \left\{
    \begin{aligned}
    & \max\left(0, \frac{g_{{x}_{k+1}}\left(\grad h({x}_{k+1}), \grad h({x}_{k+1})-\mathcal{T}_{\alpha_{k} \eta_{k}}(\xi_{k})\right)}{g_{{x}_{k}}\left( \grad h({x}_{k}), \grad h({x}_{k}) \right)}\right) & \quad \mbox{PR}_+ \\ 
    & \frac{g_{{x}_{k+1}}\left(\grad h({x}_{k+1}), \grad h({x}_{k+1})\right)}{g_{{x}_{k}}\left( \grad h({x}_{k}), \grad h({x}_{k}) \right)} &  \quad \mbox{FR}
    \end{aligned}
    \right.
    \]
\EndFor
\end{algorithmic}
\end{algorithm}

\begin{algorithm}[H]
\caption{Implementation for Riemannian Conjugate Gradient on the quotient manifold $\mathbb{C}^{n\times p}_*/\mathcal{O}_p$ with metric $g$}\label{alg:RCG_implement}
\begin{algorithmic}[1]
\Require initial iterate $\overline{x}_0 \in \mathbb{C}^{n\times p}_*$, tolerance $\varepsilon>0$, initial descent direction as $\overline{\eta}_0 = -\grad f(\overline{x}_0) = -2(\overline{x}_0\overline{x}_0^* - A)\overline{x}_0$
\For{ $k =0,1,2,\dots$}
    \State{Compute the step size $\alpha_k >0$ satisfying the strong Wolfe conditions}
    \State{Obtain the new iterate by retraction}
    \[
    \overline{x}_{k+1} = \overline{R}_{\overline{x}_k}(\alpha_k \overline{\eta}_k) = \overline{x}_k + \alpha_k \overline{\eta}_k
    \]
    \State{Compute the horizontal lift of gradient}
    \par\hskip\algorithmicindent $\overline{\xi}_{k+1} :={\grad f(\overline{x}_{k+1})} = 2(\overline{x}_{k+1} \overline{x}_{k+1}^* - A)\overline{x}_{k+1}$  
    \State{Check for convergence}
     \par\hskip\algorithmicindent if $\norm{{\overline{\xi}}_{k+1}}:=\sqrt{g_{\overline{x}_{k+1}}(\overline{\xi}_{k+1}, \overline{\xi}_{k+1})} < \varepsilon $, then break 
    \State Compute a conjugate direction by $\mbox{PR}_+$ or by $\mbox{FR}$ and vector transport
    \par\hskip\algorithmicindent $\overline{\eta}_{k+1} = -\overline{\xi}_{k+1} + \beta_{k+1} \overline{\mathcal{T}_{\alpha_{k} \eta_{k}}(\eta_{k})}_{\overline{x}_{k+1}}$
    \[  \text{ where }  \beta_{k+1} = \left\{
    \begin{aligned}
    & \max\left(0, \frac{g_{\overline{x}_{k+1}}\left(\grad f(\overline{x}_{k+1}), \grad f(\overline{x}_{k+1})-\overline{\mathcal{T}_{\alpha_{k} \eta_{k}}(\xi_{k})}_{\overline{x}_{k+1}}\right)}{g_{\overline{x}_{k}}\left( \grad f(\overline{x}_{k}), \grad f(\overline{x}_{k}) \right)}\right) & \quad \mbox{PR}_+ \\ 
    & \frac{g_{\overline{x}_{k+1}}\left(\grad f(\overline{x}_{k+1}), \grad f(\overline{x}_{k+1})\right)}{g_{\overline{x}_{k}}\left( \grad f(\overline{x}_{k}), \grad f(\overline{x}_{k}) \right)} &  \quad \mbox{FR}
    \end{aligned}
    \right.
    \]
\EndFor
\end{algorithmic}
\end{algorithm}
The following results were first proven in \cite{zheng2023riemannian}. For completeness, we include a detailed proof.
\begin{lem}\label{lem:equivalence_eta}
Let $\eta_k$ be the descent direction generated by Algorithm \ref{alg:RCG}. Then we have 
\begin{equation}
    \overline{\mathcal{T}_{\alpha_{k} \eta_{k}}(\eta_{k})}_{\overline{x}_{k+1}} = P^\mathcal{H}_{\overline{x}_k + \alpha_k \overline{\eta}_k} \left( \overline{\eta}_k \right) = \overline{\eta}_k. 
\end{equation}
\end{lem}
\begin{proof}
The first equality follows from \eqref{eqn:vector_transport_lift}.  Recall the projection formula given in proposition \ref{prop:projections_quotient}. Denote $\overline{x}_{k+1} = \overline{x}_k + \alpha_k \overline{\eta}_k$. Then  we have 
\begin{equation}
P^\mathcal{H}_{\overline{x}_k + \alpha_k \overline{\eta}_k} \left( \overline{\eta}_k \right) = \overline{\eta}_k - \overline{x}_{k+1}\Omega_k. 
\end{equation}
Hence in order to show $P^\mathcal{H}_{\overline{x}_k + \alpha_k \overline{\eta}_k} \left( \overline{\eta}_k \right) = \overline{\eta}_k$,  it is equivalent to show  the Lyapunov equation 
\begin{equation}\label{eqn:Lyapunov}
    \Omega_k \overline{x}_{k+1}^*\overline{x}_{k+1} + \overline{x}_{k+1}^*\overline{x}_{k+1} \Omega_k = \overline{x}_{k+1}^* \overline{\eta}_k- \overline{\eta}_k^* \overline{x}_{k+1}
\end{equation}
only has trivial solution $\Omega_k = 0$ for all $k \geq 0$. 

The solution $X$ to the Lyapunov equation $XE +EX = Z$ for a Hermitian $E$ is unique if $E$ is Hermitian positive-definite \cite[Section 2.2]{massart_quotient_2020}. Thus \eqref{eqn:Lyapunov} has a unique solution if $\bar x_{k+1}\in \mathbb{C}_*^{n\times p}.$
Thus we only need to show the right-hand side of the equation is zero. We prove this by induction. 

When $k =0$, the right hand side of \eqref{eqn:Lyapunov} is 
\begin{eqnarray*}
    \overline{x}_{1}^* \overline{\eta}_0- \overline{\eta}_0^* \overline{x}_{1} &=& (\overline{x}_0 + \alpha_0 \overline{\eta}_0)^* \overline{\eta}_0 -\overline{\eta}_0 ^*(\overline{x}_0 + \alpha_0 \overline{\eta}_0)  = \overline{x}_0 ^* \overline{\eta}_0  -\overline{\eta}_0^* \overline{x}_0= -2\overline{x}_0^* (\overline{x}_0 \overline{x}_0 ^* - A)\overline{x}_0 + 2\overline{x}_0 ^*(\overline{x}_0 \overline{x}_0 ^* -A^*)\overline{x}_0  = 0. 
\end{eqnarray*}

Now suppose $\overline{x}_k^* \overline{\eta}_{k-1} -\overline{\eta}_{k-1}^* \overline{x}_k =0$ and hence $P^\mathcal{H}_{\overline{x}_{k}}(\overline{\eta}_{k-1})= \overline{\eta}_{k-1}$. Then 
\begin{eqnarray*}
    \overline{x}_{k+1}^* \overline{\eta}_{k} -\overline{\eta}_{k}^* \overline{x}_{k+1} &=& (\overline{x}_k + \alpha_k \overline{\eta}_k)^*\overline{\eta}_k - \overline{\eta}_k^* (\overline{x}_k + \alpha_k \overline{\eta}_k) = \overline{x}_k^*\overline{\eta}_k - \overline{\eta}_k^* \overline{x}_k \\
    &=& \overline{x}_k^*\left(-\overline{\xi}_k + \beta_k P^\mathcal{H}_{\overline{x}_{k}}(\overline{\eta}_{k-1}) \right)  - \left(-\overline{\xi}_k + \beta_k P^\mathcal{H}_{\overline{x}_{k}}(\overline{\eta}_{k-1}) \right)^* \overline{x}_k = \overline{x}_k^*\left(-\overline{\xi}_k + \beta_k \overline{\eta}_{k-1} \right)  - \left(-\overline{\xi}_k + \beta_k \overline{\eta}_{k-1} \right)^* \overline{x}_k \\
    &=& - \overline{x}_k^* \overline{\xi}_k + \overline{\xi}_k^* \overline{x}_k = -2\overline{x}_k^*(\overline{x}_k\overline{x}_k^* - A)\overline{x}_k + 2\overline{x}_k^*(\overline{x}_k\overline{x}_k^* - A^*)\overline{x}_k= 0. 
\end{eqnarray*}
Hence $P^\mathcal{H}_{\overline{x}_{k+1}}(\overline{\eta}_{k})= \overline{\eta}_{k}$ also holds and we have proved this lemma. 
\end{proof}

We can now state our first main result:

\begin{thm}
\label{thm-1}
Algorithm \ref{alg:RCG_implement}  is equivalent to  Algorithm \ref{alg:CG}, which is the conjugate gradient method solving \eqref{min_prob_BM}, in the sense that they produce exactly the same iterates if started from the same initial point. 
\end{thm}
\begin{proof}
    By \eqref{eqn:equivalence_grad}, the gradients generated by Algorithm
    \ref{alg:CG} and  Algorithm \ref{alg:RCG_implement} are the same.  By Lemma \ref{lem:equivalence_eta} and the equivalence between the Riemannian metric on $\mathbb{C}^{n\times p}_*$ and the inner product on $\mathbb{C}^{n\times p}$, we see that $\beta_k$ generated by these two algorithms are also equivalent. Hence the conjugate directions are also the same. So the two algorithms generate the same iterates. 
\end{proof}

\section{The Convergence of the Fletcher-Reeves Conjugate Gradient Method}
\label{sec-proof}

In this section, we will prove that the Riemannian 
Fletcher-Reeves Conjugate Gradient method converges to a stationary point thus Algorithm \ref{alg:CG}  also converges by the equivalence Theorem \ref{thm-1}.

The discussion in this section follows the same lines as in standard convergence theory, e.g.,  \cite{sato2015new}. The cost function and vector transport considered in this paper satisfy the conditions for convergence analysis in \cite{sato2015new}. \textcolor{black}{Many results in this section are standard convergence results for a line search method, see \cite {nocedal1999numerical}.} For completeness, we include the full proof.

Let $\eta_k \in T_{x_k}\mathbb{C}^{n\times p}_*/\mathcal{O}_p$ be a descent direction. Define the angle $\theta_k$ between $-\grad h(x_k)$ and $\eta_k$ by 
\begin{equation}\label{eqn:theta_k}
    \cos \theta_k = - \frac{g_{x_k }\left(\grad h(x_k), \eta_k \right)}{\norm{\grad h(x_k) }_{x_k}\norm{\eta_k}_{x_k}}.
\end{equation}

Let $\mathcal{L} := \{x \in \mathbb{C}^{n\times p}_*/\mathcal{O}_p: 0 \leq h(x)  \leq h(x_0) \}$ and  $\pi^{-1}(\mathcal{L}) = \{\overline{x} \in \mathbb{C}^{n\times p}_*: 0 \leq f(\overline{x})  \leq f(\overline{x}_0) \}$. We can show that $\pi^{-1}(\mathcal{L})$ is bounded. 
\begin{lem}
    There is a constant $C$ such that $\|\bar x\|_F\leq C,\quad \forall \bar x \in \pi^{-1}(\mathcal L)$.
\end{lem}
\begin{proof}
Assume it is not true, then $\forall n\in \mathbbm N, \exists \bar x_n\in \pi^{-1}(\mathcal L)$ such that $\|\bar x_n\|_F\geq n.$ Let $y_n=\frac{\bar x_n}{\|\bar x_n\|_F}$, then $\|y_n\|_F=1$ and $\bar x_n= \|\bar x_n\|_F y_n=a_n y_n$ with $a_n\geq n$. Thus
$f(\bar x_n)=\frac12 \|a^2_n y_n y_n^*-A\|_F^2\to \infty$ since $a_n\to \infty$ and $\|y_n\|_F=1.$ On the other hand, $\bar x_n \in \pi^{-1}(\mathcal L)$ implies that $f(\bar x_n)$ should be bounded, which is a contradiction. 
    \end{proof}

\begin{lem} \label{lem:gradient_lipschitz}
The Riemannian gradient of $f$, i.e., 
    $\grad f(\overline{x}) = 2(\overline{x}\overline{x}^* - A)\overline{x}$ is  Lipschitz continuous on $\pi^{-1}(\mathcal{L})$. That is, there exists a constant $L >0$ such that
    \begin{equation}
        \norm{\grad f(\overline{y}) - \grad f(\overline{x})}_F \leq L\norm{\overline{y} - \overline{x}}_F, \quad \text{ for all } \overline{x}, \overline{y} \in \pi^{-1}(\mathcal{L}).
    \end{equation}
\end{lem}
\begin{proof}
It suffices to show that $q: \overline{x} \mapsto \overline{x}\overline{x}^*\overline{x}$ is Lipschitz continuous on $\pi^{-1}(\mathcal{L})$.    Let $\overline{x},\overline{y} \in \pi^{-1}(\mathcal{L})$. Then  $\norm{\overline{x}}_F \leq C, \norm{\overline{y}}_F \leq C$ by Lemma $\ref{lem:gradient_lipschitz}$. 
\begin{eqnarray*}
    \norm{q(\overline{x}) - q(\overline{y})}_F &=& \norm{\overline{x}\overline{x}^*\overline{x} - \overline{y}\overline{y}^*\overline{y} }_F =\norm{\overline{x}\overline{x}^*\overline{x} - \overline{x}\overline{x}^*\overline{y} +  \overline{x}\overline{x}^*\overline{y} - \overline{y}\overline{y}^*\overline{y} }_F \\
    &\leq& \norm{\overline{x}\overline{x}^*\overline{x} - \overline{x}\overline{x}^*\overline{y}}_F  +\norm{  \overline{x}\overline{x}^*\overline{y} - \overline{y}\overline{y}^*\overline{y} }_F =\norm{\overline{x}\overline{x}^*\overline{x} - \overline{x}\overline{x}^*\overline{y}}_F  +\norm{  \overline{x}\overline{x}^*\overline{y} -\overline{y}\overline{x}^*\overline{y} +\overline{y}\overline{x}^*\overline{y}  - \overline{y}\overline{y}^*\overline{y} }_F \\ 
    &\leq&  \norm{\overline{x}\overline{x}^*\overline{x} - \overline{x}\overline{x}^*\overline{y}}_F  +\norm{  \overline{x}\overline{x}^*\overline{y} -\overline{y}\overline{x}^*\overline{y}}_F +\norm{\overline{y}\overline{x}^*\overline{y}  - \overline{y}\overline{y}^*\overline{y} }_F \\ 
    &\leq& \norm{\overline{x}\overline{x}^*}\norm{\overline{x} - \overline{y}}_F + \norm{\overline{x}-\overline{y}}_F \norm{\overline{x}^*}_F\norm{\overline{y}}_F + \norm{\overline{y}}_F\norm{\overline{x}^* - \overline{y}^*}_F \norm{\overline{y}}_F \leq  3C^2 \norm{\overline{x} - \overline{y}}_F.
\end{eqnarray*}
\end{proof}

\begin{thm}[Zoutendijk’s theorem on manifold]\label{thm:zoutendijk}
Let $\eta_k$ be a descent direction and let $\alpha_k $ satisfy the strong Wolfe conditions \eqref{wolfe_1} and \eqref{wolfe_2}. Then for the cost function $h$ defined in \ref{eqn:cost_function_quotient}, the following series converges.
\[
\sum_k^\infty \cos^2\theta_k \norm{\grad h(x_k)}_{x_k}^2 < \infty.
\]
\end{thm}
\begin{proof}
    From the strong Wolfe condition \eqref{wolfe_2} we have 
    \begin{eqnarray*}
         (c_2 -1 ) g_{x_k}(\grad h(x_k), \eta_k) 
        & \leq&  g_{x_{k+1}}\left((\grad h(R_{x_k}(\alpha_k \eta_k), \D R_{x_k}(\alpha_k \eta_k)[\eta_k]\right) - g_{x_k}\left( \grad h(x_k), \eta_k\right)\\
        & = & g_{\overline{x}_{k+1}}\left(\grad f(\overline{x}_k+\alpha_k \overline{\eta}_k), P^\mathcal{H}_{\overline{x}_k + \alpha_k \overline{\eta}_{k}}(\overline{\eta}_k)\right) - g_{\overline{x}_k}\left( \grad f(\overline{x}_k), \overline{\eta}_k\right) \\
        &=& g_{\overline{x}_{k+1}}\left(\grad f(\overline{x}_k+\alpha_k \overline{\eta}_k), \overline{\eta}_k\right) - g_{\overline{x}_k}\left( \grad f(\overline{x}_k), \overline{\eta}_k\right). 
    \end{eqnarray*}
    
    Notice that our Riemannian metric $g$ is simply the inner product on the Euclidean space $\mathbb{C}^{n\times p}$, hence
    
    \begin{equation}
        g_{\overline{x}_{k+1}}\left(\grad f(\overline{x}_k+\alpha_k \overline{\eta}_k), \overline{\eta}_k\right) - g_{\overline{x}_k}\left( \grad f(\overline{x}_k), \overline{\eta}_k\right) =  \ip{\grad f(\overline{x}_k+\alpha_k \overline{\eta}_k) -  \grad f(\overline{x}_k),\overline{\eta}_k } .
    \end{equation}
    
    From Lemma \ref{lem:gradient_lipschitz} we know
    
    \[
     \ip{\grad f(\overline{x}_k+\alpha_k \overline{\eta}_k) -  \grad f(\overline{x}_k),\overline{\eta}_k } \leq  \alpha_k L \norm{\overline{\eta}_k}_F^2.
    \]
    
    Hence for any $k$ we have 
    \begin{equation}\label{eqn:alpha}
        \alpha_k \geq \frac{(c_2 - 1)g_{x_k}(\grad h(x_k),\eta_k)}{L \norm{\overline{\eta}_k}^2_F}.
    \end{equation}
    
    Now it follows from \eqref{wolfe_1} and \eqref{eqn:alpha} that 
    \begin{eqnarray*}
        0 \leq h(x_{k+1}) &\leq& h(x_k) + c_1 \alpha_k g_{x_k}(\grad h(x_k),\eta_k) \\    &\leq & h(x_k) - \frac{c_1(1-c_2)}{L} \cos^2\theta_k \norm{\grad h(x_k)}_{x_k}^2  \leq  h(x_0) - \frac{c_1(1-c_2)}{L} \sum_{j=0}^k\cos^2\theta_j \norm{\grad h(x_j)}_{x_j}^2. 
    \end{eqnarray*}
    
    Hence 
    \begin{equation}
        \sum_{k=0}^\infty \cos^2\theta_k \norm{\grad h(x_k)}_{x_k}^2 \leq \frac{L}{c_1 (1-c_2)} h(x_0) < \infty.  
    \end{equation}
\end{proof}

\begin{lem}
If using  Fletcher-Reeves  method in Algorithm $\ref{alg:RCG}$, then for $0 < c_1<c_2 < 1/2$, the search direction $\eta_k$  is a descent direction satisfying 
\begin{equation}\label{eqn:cos}
    -\frac{1}{1-c_2} \leq \frac{g_{x_k}(\grad h(x_k), \eta_k)}{\norm{\grad h(x_k) }_{x_k}^2} \leq \frac{2c_2 - 1}{ 1- c_2}. 
\end{equation}
\end{lem}

\begin{proof}
    We prove it by induction on $k$.
    
    When $k = 0$,  \eqref{eqn:cos}  holds since
    \[
        \frac{g_{x_0}(\grad h(x_0), \eta_0)}{\norm{\grad h(x_0) }_{x_0}^2} = \frac{g_{x_0}(\grad h(x_0), -\grad h(x_0))}{\norm{\grad h(x_0) }_{x_0}^2} = -1.
    \]
    
    Now suppose \eqref{eqn:cos} holds for some $k\geq 0$.

    Recall that we use differentiated retraction as our vector transport:
    \[
    \mathcal{T}_{\alpha_k \eta_k}(\eta_k) = \D R_{x_k}(\alpha_k \eta_k)[\eta_k].
    \]
    And the  $\beta_{k+1}$ in Fletcher-Reeves method is defined as 
    \[
    \beta_{k+1} = \frac{g_{{x}_{k+1}}\left(\grad h({x}_{k+1}), \grad h({x}_{k+1})\right)}{g_{{x}_{k}}\left( \grad h({x}_{k}), \grad h({x}_{k}) \right)}.
    \]
    
    Hence the middle term in \eqref{eqn:cos} for $k+1$ is 
    \begin{eqnarray}
    \frac{g_{x_{k+1}}(\grad h(x_{k+1}), \eta_{k+1})}{\norm{\grad h(x_{k+1}) }_{x_{k+1}}^2}  \notag&=&  \frac{g_{x_{k+1}}\left( \grad h(x_{k+1}), - \grad h(x_{k+1})+ \beta_{k+1} \mathcal{T}_{\alpha_k \eta_k}(\eta_k)  \right)}{\norm{\grad h(x_{k+1}) }_{x_{k+1}}^2} \\
     \notag&=& \frac{g_{x_{k+1}}\left( \grad h(x_{k+1}), - \grad h(x_{k+1})+ \beta_{k+1} \D R_{x_k}(\alpha_k \eta_k)[\eta_k]) \right)}{\norm{\grad h(x_{k+1}) }_{x_{k+1}}^2}  \\
    \label{eqn:1}&=& -1 + \frac{g_{x_{k+1}}\left(\grad h(x_{k+1})), \D R_{x_k}(\alpha_k \eta_k)[\eta_k] \right) }{\norm{\grad h(x_{k}) }_{x_{k}}^2}.
    \end{eqnarray}
    
    From the strong Wolfe condition \eqref{wolfe_2} we have  
    \begin{equation} \label{eqn:2}
        c_2 g_{x_k}(\grad h(x_k), \eta_k) \leq g_{x_{k+1}}(\grad h(x_{k+1}), \D R_{x_k}(\alpha_k \eta_k)[\eta_k]) \leq  -c_2 g_{x_k}(\grad h(x_k), \eta_k).
    \end{equation}

    Hence from \eqref{eqn:1} and \eqref{eqn:2} we have 
    \[
    -1+ c_2 \frac{g_{x_k}(\grad h(x_k), \eta_k)}{\norm{\grad h(x_{k}) }_{x_{k}}^2} \leq \frac{g_{x_{k+1}}(\grad h(x_{k+1}), \eta_{k+1})}{\norm{\grad h(x_{k+1}) }_{x_{k+1}}^2} \leq -1 - c_2 \frac{g_{x_k}(\grad h(x_k), \eta_k)}{\norm{\grad h(x_{k}) }_{x_{k}}^2}. 
    \]
    And the result \eqref{eqn:cos} follows from the induction hypothesis.

\end{proof}

\begin{thm}
For cost function $h$ in \eqref{eqn:cost_function_quotient},  the Algorithm $\ref{alg:RCG}$ with  Fletcher-Reeves method  generates iterates $x_k$ such that 
\begin{equation}\label{eqn:liminf_grad}
    \liminf_{k \to \infty} \norm{\grad h(x_k)}_{x_k } = 0.
\end{equation}
\end{thm}
\begin{proof}
    If $\grad h(x_k) =0 $ for some $k = k_0$. Then $\grad h(x_k) = 0$ for all $k\geq k_0$. 
    
    So we consider $\grad h(x_k) \neq 0$ for all $k$. We shall prove  \eqref{eqn:liminf_grad} by contradiction. Suppose \eqref{eqn:liminf_grad} does not hold. Then there exists a constant $c >0$ such that  
    \begin{equation}\label{eqn:contradiction_assumption}
        \norm{\grad h(x_k)}_{x_k} \geq c >0,\quad \forall k\geq 0. 
    \end{equation}
    
    From \eqref{eqn:theta_k} and \eqref{eqn:cos}  we have
    
    \begin{equation}
        \cos \theta_k \geq \frac{1-2c_2}{1-c_2} \frac{\norm{\grad h(x_k)}_{x_k}}{\norm{\eta_k}_{x_k}}. 
    \end{equation}
    It follows by Theorem \ref{thm:zoutendijk} that the following series converges. 
    
    \begin{equation} \label{eqn:series_converge}
        \sum_{k=0}^\infty \frac{\norm{\grad h(x_k)}_{x_k}^4}{\norm{\eta_k}_{x_k}^2} < \infty. 
    \end{equation}
    
    For $k\geq 1$, the strong Wolfe condition \eqref{wolfe_2} and \eqref{eqn:cos} gives rise to 
    \begin{eqnarray*}
    \abs{g_{x_k} \left(\grad h(x_k), \mathcal{T}_{\alpha_{k-1} \eta_{k-1}}  (\eta_{k-1}) \right)}  &\leq& - c_2 g_{x_{k-1}} \left( \grad h(x_{k-1}), \eta_{k-1}\right) \leq \frac{c_2}{1-c_2} \norm{ \grad h(x_{k-1}) }_{x_{k-1}}^2. 
    \end{eqnarray*}   
    
    Hence we have the following recurrence equation for $\norm{\eta_k}^2_{x_k}$. 
    \begin{eqnarray}
    \norm{\eta_k}^2_{x_k} \notag&=& \norm{-\grad h(x_k) + \beta_k \mathcal{T}_{\alpha_{k-1} \eta_{k-1}}(\eta_{k-1})}^2_{x_k}  \\
    \notag&\leq & \norm{\grad h(x_k)}^2_{x_k} + 2 \beta_k \abs{g_{x_k}\left( \grad h(x_k), \mathcal{T}_{\alpha_{k-1} \eta_{k-1}}(\eta_{k-1}) \right)} + \beta_k^2 \norm{\mathcal{T}_{\alpha_{k-1} \eta_{k-1}}(\eta_{k-1})}^2_{x_{k}} \\
    \notag &\leq & \norm{\grad h(x_k)}^2_{x_k} + \frac{2c_2 }{1-c_2}\beta_k \norm{ \grad h(x_{k-1}) }_{x_{k-1}}^2 + \beta_k^2 \norm{\mathcal{T}_{\alpha_{k-1} \eta_{k-1}}(\eta_{k-1})}^2_{x_{k}} \\
    \notag &=& \norm{\grad h(x_k)}^2_{x_k} + \frac{2c_2 }{1-c_2} \norm{ \grad h(x_{k}) }_{x_{k}}^2 + \beta_k^2 \norm{\mathcal{T}_{\alpha_{k-1} \eta_{k-1}}(\eta_{k-1})}^2_{x_{k}} \\
    \label{eqn:norm_eta_k}  &=& \frac{1+c_2}{1-c_2}  \norm{ \grad h(x_{k}) }_{x_{k}}^2 + \beta_k^2 \norm{\mathcal{T}_{\alpha_{k-1} \eta_{k-1}}(\eta_{k-1})}^2_{x_{k}} .
    \end{eqnarray}
    
    Recall that we use differentiated retraction as our vector transport:
    \[
    \mathcal{T}_{\alpha_{k-1} \eta_{k-1}}(\eta_{k-1}) = \D R_{x_{k-1}}(\alpha_{k-1} \eta_{k-1})[\eta_{k-1}] = \D \pi(\overline{x}_{k-1} + \alpha_{k-1} \overline{\eta}_{{k-1}}) \left[ P^\mathcal{H}_{\overline{x}_{k-1} + \alpha_{k-1} \overline{\eta}_{k-1}}(\overline{\eta}_{k-1}) \right].
    \]
    
    Hence 
    \begin{eqnarray*}
     \norm{\mathcal{T}_{\alpha_{k-1} \eta_{k-1}}(\eta_{k-1})}^2_{x_{k}} &=&  g_{x_{k}}\left(\mathcal{T}_{\alpha_{k-1} \eta_{k-1}}(\eta_{k-1}),\mathcal{T}_{\alpha_{k-1} \eta_{k-1}}(\eta_{k-1}) \right)= g_{\overline{x}_{k}}\left(\overline{\mathcal{T}_{\alpha_{k-1} \eta_{k-1}}(\eta_{k-1})}_{\overline{x}_k},\overline{\mathcal{T}_{\alpha_{k-1} \eta_{k-1}}(\eta_{k-1})}_{\overline{x}_k}\right)  \\
     \notag&=& g_{\overline{x}_{k}}\left(P^\mathcal{H}_{\overline{x}_{k-1} + \alpha_{k-1} \overline{\eta}_{k-1}}(\overline{\eta}_{k-1}),P^\mathcal{H}_{\overline{x}_{k-1} + \alpha_{k-1} \overline{\eta}_{k-1}}(\overline{\eta}_{k-1})\right) = g_{\overline{x}_{k-1}}\left(\overline{\eta}_{k-1},\overline{\eta}_{k-1}\right) =\norm{{\eta}_{k-1}}^2_{x_{k-1}}. 
    \end{eqnarray*}
    Hence \eqref{eqn:norm_eta_k} becomes the following recurrence formula for $ \norm{\eta_k}^2_{x_{k}}$.
    \begin{equation}
        \norm{\eta_k}^2_{x_{k}} \leq \frac{1+c_2}{1-c_2} \norm{\grad h(x_k)}^2_{x_k}  + \beta_k^2\norm{{\eta}_{k-1}}^2_{x_{k-1}}. 
    \end{equation}
    
    By recursively using \eqref{eqn:norm_eta_k} and recall the definition of $\beta_k$ in Fletcher-Reeves method  we obtain 
    \begin{eqnarray}\label{eqn:norm_eta_bound}
        \norm{\eta_k}^2_{x_k} \notag &\leq &  \frac{1+c_2}{1-c_2} \left( \norm{\grad h(x_k)}^2_{x_k} + \beta_k^2 \norm{\grad h(x_{k-1})}^2_{x_{k-1}}+ \dots + \beta_k^2\beta_{k-1}^2 \dots \beta_2^2 \norm{\grad h(x_1)}^2_{x_1}\right) \\ 
       \notag && + \beta_k^2\beta_{k-1}^2 \dots \beta_0^0 \norm{\eta_0}^2_{x_0} \\
       \notag &=& \frac{1+c_2}{1-c_2} \norm{\grad h(x_k)}^4_{x_k} \left( \norm{\grad h(x_k)}^{-2}_{x_k} +  \norm{\grad h(x_k)}^{-2}_{x_{k-1}}+\dots +  \norm{\grad h(x_k)}^{-2}_{x_1} \right)  \\
       \notag && +\norm{\grad h(x_k)}^4_{x_k} \norm{\grad h(x_0)}^{-2}_{x_0} \\
       \notag &<& \frac{1+c_2}{1-c_2}  \norm{\grad h(x_k)}^{4}_{x_k} \sum_{j=0}^{k} \norm{\grad h(x_j) }^{-2}_{x_{j}}  \leq \frac{1+c_2}{1-c_2}  \norm{\grad h(x_k)}^{4}_{x_k}  \frac{k+1}{c^2},
    \end{eqnarray}
    where we have used the contradiction assumption \eqref{eqn:contradiction_assumption} in the last inequality. \eqref{eqn:norm_eta_bound} results in the divergence of the following series. 
    \begin{equation}
        \sum_{k=0}^\infty \frac{\norm{\grad h(x_k)}^4_{x_k}}{\norm{\eta_k}_{x_k}^2} \geq c^2\frac{1-c_2}{1+c_2}\sum_{k=0}^\infty \frac{1}{k+1} = \infty. 
    \end{equation}
    This contradicts to \eqref{eqn:series_converge} and hence we have completed the proof.
\end{proof}

In general, it is more difficult to prove the convergence of the Riemannian $\mbox{PR}_+$ CG method. It is possible to extend the convergence proof of  $\mbox{PR}_+$ CG method in \cite{gilbert1992global} to Riemannian $\mbox{PR}_+$ CG method, but it is beyond the scope of this paper.

\section{Coordinate Riemannian Gradient Descent (CRGD)}
\label{sec-CRGD}
\textcolor{black}{
The orthogonalization-free methods are preferred for large scale problems. For much larger problems, the coordinate descent method is favored, since the full gradient can be too large to even store. For instance, the coordinate gradient descent method for finding leading eigenvalue in \cite{li2019coordinatewise} is the coordinate descent method for minimizing \eqref{min_prob_BM} with rank $p=1$. In this section, following the same Riemannian manifold notation as in previous sections, we show that the a Riemmanian coordinate descent method is also equivalent to the coordinate descent method for minimizing \eqref{min_prob_BM} with any rank $p>0$, which is the generalization of the algorithm in  \cite{li2019coordinatewise}.
}

\textcolor{black}{
In \cite{gutman_coordinate_2020}, a method called the tangent subspace descent method was proposed: this method generalized the block coordinate descent method to manifold settings. Instead of updating the full gradient at each iteration, the tangent direction in each update is a projected vector of the full Riemannian gradient to a subspace of the tangent space by some subspace selection rule $P_k$.  In the specific case of $\mathbb{C}^{n\times p}_*/\mathcal{O}_p$ considered in this paper, this method is written as Algorithm $\ref{alg:Coordinate_RGD}$ and we denote it as Coordinate Riemannian Gradient Descent (CRGD). }

Since the horizontal lift of  $\grad h(x_k)$ is a $n$-by-$p$ matrix, we can simply choose the subspace selection rule by cyclically selecting the $N$-column block of the $n$-by-$p$ matrix $\grad f(\overline{x}_k)$. Let $M_k$ denote the mask that evaluates the $k$-th $N$-column block of a $n$-by-$p$ matrix cyclically. That is, if $Z$ is a $n$-by-$p$ matrix, then
\begin{equation} \label{eqn:cyclic_projection}
    M_k(Z) = Z_{kN+1:(k+1)N,:} 
\end{equation}
where $Z_{kN+1:(k+1)N,:}$ denotes the $N$-by-$p$ matrix that takes the $(kN+1)$-th to $(k+1)N$-th columns of $Z$.  And the index that exceeds the matrix range is understood as modulo by the matrix size, namely, cyclically. Then our update to $\overline{x}_k$ is  written through the following 
\begin{equation}
    \overline{x}_{k+1} = \overline{R}_{\overline{x}_k}(\alpha M_k(\grad f(\overline{x}_k))),
    \label{CRGD}
\end{equation}
where $\alpha$ is a constant step size. 

\textcolor{black}{
With the simple retraction as in Section \ref{sec:retraction}, \eqref{CRGD} simply reduces to 
\begin{equation} \overline{x}_{k+1} = \overline{x}_{k+1}- \alpha M_k(2 (\overline{x}_k\overline{x}_k^*-A)\overline{x}_k)     \label{CRGD2}.
\end{equation}
Notice that \eqref{CRGD2} with $p=1$ and $N=1$ reduces to the 
coordinate descent method for the leading eigenvalue 
 in \cite{li2019coordinatewise}. In particular, if $p=1$ and we set $N=1$ and $P_k$ in Algorithm \ref{alg:Coordinate_RGD} to be $M_k$, defined in \eqref{eqn:cyclic_projection}, then Algorithm  \ref{alg:Coordinate_RGD} is equivalent to Algorithm 2 in \cite{li2019coordinatewise}. 
 So the generalization of the method  in \cite{li2019coordinatewise} to top $p$ eigenvalues can be equivalently written as \eqref{CRGD2} or \eqref{CRGD}, which is a Riemannian coordinate descent method.
}

To take the advantage of CRGD to solve large-scaled problems, one should implement it through compact implementation. That is, each update should only depend on the block size $N$ and should be independent of the problem size $n$. In the case of eigenvalue problem, $f(\overline{x}) = \frac{1}{2}\norm{\overline{x}\overline{x}^* - A}_F^2$. If we assume that $A$ is a sparse matrix such that we can achieve $M_{k}(Av)$ in $O(N)$, then we can indeed achieve a compact implementation of CRGD as in Algorithm \ref{alg:Compact_Implementation_Coordinate_RGD}.
\begin{algorithm}[ht]
\caption{Coordinate Riemannian gradient descent (CRGD) on the quotient manifold $\mathbb{C}^{n\times p}_*/\mathcal{O}_p$ with metric $g$}
\label{alg:Coordinate_RGD}
\begin{algorithmic}[1]
\Require initial iterate ${x}_0 \in \mathbb{C}^{n\times p}_*/\mathcal{O}_p$, tolerance $\varepsilon>0$, tangent vector $\xi_0 = -\grad h(x_0)$, \textcolor{black}{subspace selection rule $P_k$}, $\delta_0 := P_0(\xi_0)$, stepsize $\alpha >0$. 
\For{ $k =0,1,2,\dots$}
    \State{Obtain the new iterate by retraction}
    \[
    x_{k+1} = R_{x_k}(\alpha \delta_k)
    \]
    \State{Compute the projection of ${\xi}_{k+1} := - {\grad h(x_{k+1})}$ to a subspace of $T_{x_{k+1}} \mathbb{C}^{n\times p}/\mathcal{O}_p$}
    \par\hskip\algorithmicindent $\delta_{k+1} := P_{k+1}({\xi}_{k+1}) $ 
    \State{Check for convergence}
     \par\hskip\algorithmicindent if $\norm{{\delta}_{k+1}}:=\sqrt{g_{{x}_{k+1}}({\delta}_{k+1}, {\delta}_{k+1})} < \varepsilon $, then break 
\EndFor
\end{algorithmic}
\end{algorithm}

\begin{algorithm}[ht]
\caption{Compact implementation for cyclic coordinate Riemannian gradient descent on the quotient manifold $\mathbb{C}^{n\times p}_*/\mathcal{O}_p$ with metric $g$}\label{alg:Compact_Implementation_Coordinate_RGD}
\begin{algorithmic}[1]
\Require initial iterate $\overline{x}_0 \in \mathbb{C}^{n\times p}_*$, $\overline{\eta}_0 = -\grad f(\overline{x}_0) \in \mathbb{C}^{n\times p}$, first $N$ columns of $\overline{\eta}_0$: $\overline{\delta}_0 = \mathcal{M}_0(\overline{\eta}_0)$, $a_0=x_0^*x_0$, $b_0 = \delta_0^* x_0$, $c_0 = \delta_0^* \delta_0$, stepsize $\alpha >0$, $s_0 = a_0 + \alpha b_0 + \alpha b_0^* + \alpha^2 c_0$,  tolerance $\varepsilon>0$.
\For{ $k =0,1,2,\dots$}
    \State{Obtain the new iterate by retraction}
    \[
    \overline{x}_{k+1} = \overline{R}_{\overline{x}_k}(\alpha \overline{\delta}_k) = \overline{x}_k + \alpha \overline{\delta}_k
    \]
    \State{Cyclically compute the next $N$ columns of $\overline{\eta}_{k+1} = - \grad f(\overline{x}_{k+1})$}
    \par\hskip\algorithmicindent $\overline{\delta}_{k+1} := -2 M_{k+1}(\overline{x}_k s_k ) -2 \alpha M_{k+1}(\overline{\delta}_k s_k) +2M_{k+1}(A\overline{x}_k) + 2\alpha M_{k+1}(A\overline{\delta}_k)$
    \State{Check for convergence}
     \par\hskip\algorithmicindent if $\norm{{\overline{\delta}}_{k+1}}:=\sqrt{g_{\overline{x}_{k+1}}(\overline{\delta}_{k+1}, \overline{\delta}_{k+1})} < \varepsilon $, then break
    \State{Compute and update $a_{k+1},b_{k+1},c_{k+1}$}
    \par\hskip\algorithmicindent 
    \[
    \begin{aligned}
    & a_{k+1} = a_k + \alpha \overline{x}_k^* \overline{\delta}_k + \alpha \overline{\delta}_k^* x_k + \alpha^2 \overline{\delta}_k^* \overline{\delta}_k  \\ 
    & b_{k+1} =  \overline{\delta}_{k+1}^* \overline{x}_{k+1} \\
    & c_{k+1} = \overline{\delta}_{k+1}^* \delta_{k+1}
    \end{aligned}
    \]
    \State{Compute temporary variable $s_{k+1} \in \mathbb{C}^{p \times p}$}
    \par\hskip\algorithmicindent 
    $s_{k+1} = a_{k+1} + \alpha b_{k+1} + \alpha b_{k+1}^* + \alpha^2 c_{k+1} $
\EndFor
\end{algorithmic}
\end{algorithm}

\section{Numerical Experiments}
\label{sec-tests}
\textcolor{black}{
The numerical performance of the simple CG methods \eqref{BMCG} has been well studied in the literature, e.g., see \cite{gao2022triangularized} for a comparison with other orthogonalization-free methods. In general, the performance of \eqref{BMCG} for solving \eqref{min_prob_BM} depends on the spectrum of the matrix $A$.
For completeness, in this section we verify the numerical performance of the simple CG methods \eqref{BMCG} on  large matrices $A$. 
}

\subsection{Real symmetric PSD matrices}

We consider two types of matrices $A$. The first type is a 2D Laplacian matrix, which has a nearly uniform eigenvalue gap for a few top eigenvalues. Consider the discretization of a 2D Poisson equation with homogeneous Dirichlet boundary conditions on $[0,1]\times [0,1]$ using $m$-by-$m$ interior grid points. Then the matrix representing the Laplacian operator is a 2D Laplacian matrix $A$ of size $m^2$-by-$m^2$ given as
\begin{equation}\label{eqn:2d_laplacian_matrix}
    A = \frac{1}{\Delta x^2}K \otimes I_m + I_m\otimes \frac{1}{\Delta y^2} K,
\end{equation}
where $\Delta x = \Delta y = \frac{1}{m+1}$ and $K$ is a $m$-by-$m$ tridiagonal matrix.
\begin{equation}\label{eqn:laplacian_matrix}
    K = \bmat{
     2 & -1 &       &        &       &   \\ 
    -1 & 2  & -1    &        &       &   \\
       & -1 &  2    & -1     &       &   \\
       &    &\ddots & \ddots & \ddots&   \\
       &    &       & -1     & 2     & -1 \\
       &    &       &        & -1    & 2
       }
\end{equation}

The second type is constructed by eigenvalue decomposition $A=V\Lambda V^{-1}$ where eigenvectors $V$ are given by discrete cosine transform. We  assign $\Lambda$ so that the eigenvalues 
 $\lambda_i$ have four types of distribution of eigenvalues, similar to the numerical experiments considered in \cite{gao2022triangularized} but with a much larger matrix size: 
\begin{enumerate}
    \item (random) $\lambda_ i \sim \abs{\mathcal{N}(0,1)}$, where $\mathcal{N}(0,1)$ is standard normal distribution.  
    \item (uniform) $\lambda_i = 1 - \frac{i-1}{n}, \quad 1\leq i \leq r.$
    \item (u-shape) $\lambda_1 = \frac{14}{16}, \lambda_2 = \frac{10}{16}, \lambda_3 = \frac{8}{16}, \lambda_4 = \frac{7}{16},\lambda_5 = \frac{5}{16}, \lambda_i = \frac{1}{16}.$
    \item (logarithm) $\lambda_i = \frac{2^{1+\lfloor \log_2 n \rfloor}}{n} \frac{1}{2^i}, \quad 1\leq i \leq r.$
\end{enumerate}

We first compare the simple CG methods \eqref{BMCG} with the TriOFM method in \cite{gao2021global} for a 2D discrete Laplacian matrix, shown in Figure \ref{fig-laplacian-1}.

 {Next, we compare TriOFM, CG and LOBPCG for different distributed eigenvalues. We use Algorithm 1  in \cite{doi:10.1137/17M1129830} as the orthogonalization-free LOBPCG method in numerical tests. } The comparison is shown for randomly distributed eigenvalues in Figure \ref{fig:rand-dstribution}, uniformly distributed eigenvalues in Figure \ref{fig:dct_uniform}, U-shape distribution of eigenvalues in Figure \ref{fig:dct_ushape},
and log distribution of eigenvalues in Figure \ref{fig:log-2}. 
{In all these comparisions, the orthogonalization-free LOBPCG method is the most efficient one.}
Notice that the simple CG-PR method is much less efficient than the TriOFM method for the log distribution of eigenvalues. However, this slowness is due to the eigenvalue gap between $\sigma_p$ and $\sigma_{p+1}$. In Figure \ref{fig:log-3}, the top $p$ eigenvalues with $p=5$ have a log distribution but the gap between  $\sigma_p$ and $\sigma_{p+1}$ is enlarged by shifting the top $p$ eigenvalues from the same matrix in Figure \ref{fig:log-2}, and we observe that the simple CG-PR method is efficient in this scenario. In other words,  the matrix in Figure \ref{fig:log-2} has
eigenvalues $\lambda_1\geq \lambda_2\geq \cdots\geq \lambda_n$,
and the matrix 
in Figure \ref{fig:log-3} has
eigenvalues $\lambda_1+C\geq \lambda_2+C\geq \dots \geq \lambda_p+C\geq \lambda_{p+1}\geq \cdots\geq \lambda_n$. \textcolor{blue}{}

\begin{figure}[H]
\centering
\subfigure[Relative error vs iteration]{
\includegraphics[width=0.46\textwidth]{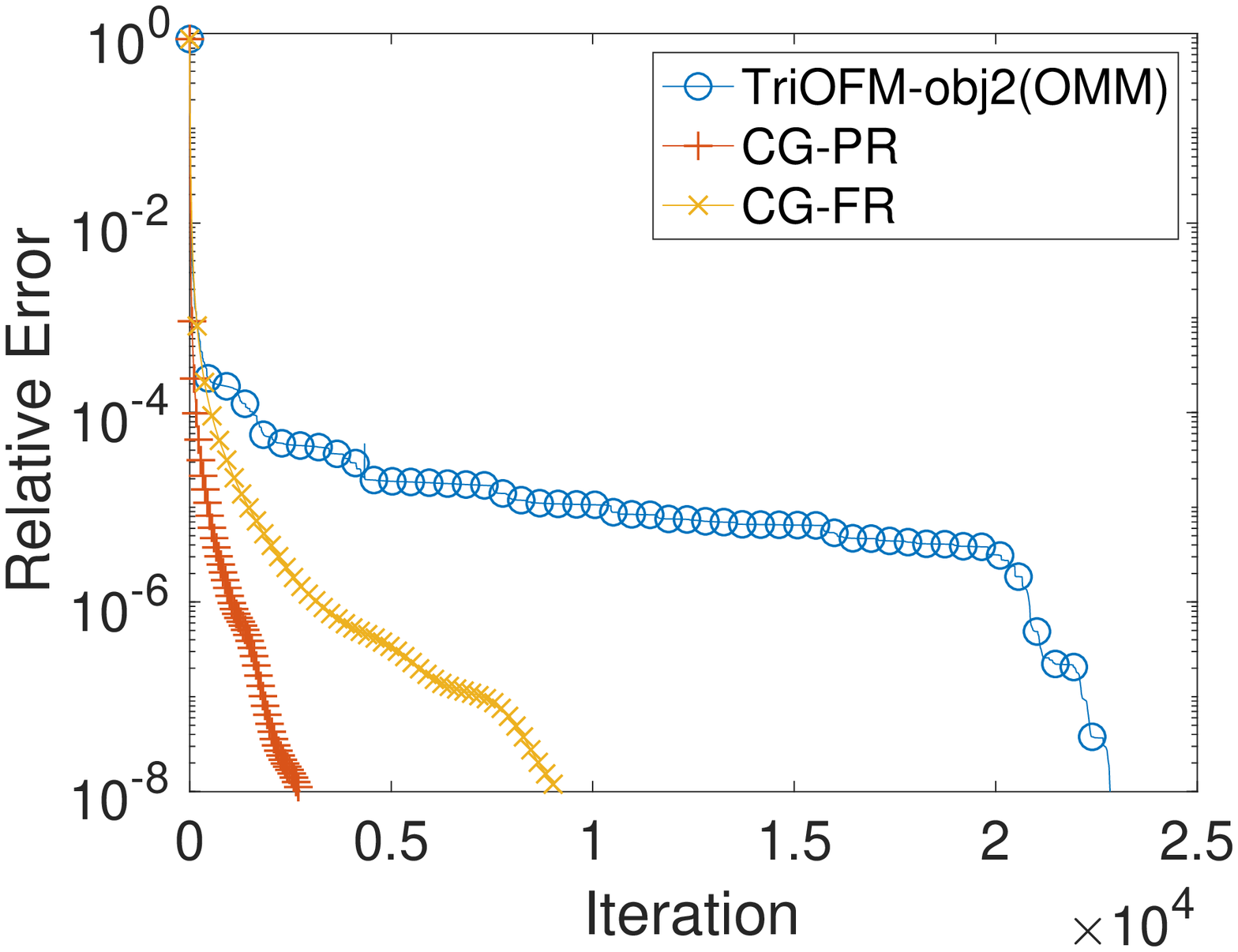}
}
\subfigure[Relative error vs CPU time]{
\includegraphics[width=0.46\textwidth]{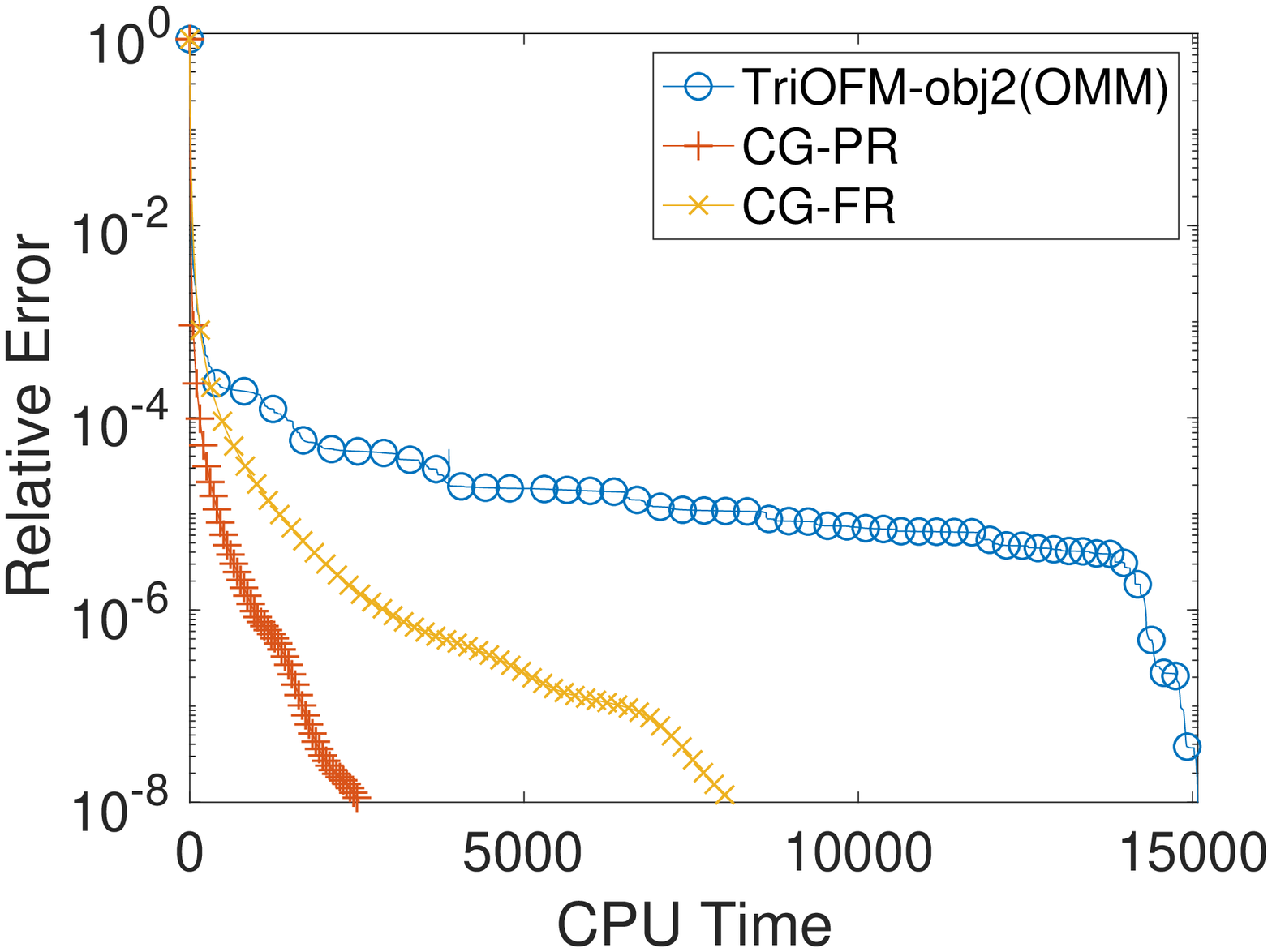}
}
\caption{Comparison for computing the top-10 eigenvalues of a 2D Laplacian matrix of size $10^6\times 10^6$.}
\label{fig-laplacian-1}
\end{figure}

\begin{figure}[H]
\centering
\subfigure[Relative error vs iteration]{
\includegraphics[width=0.46\textwidth]{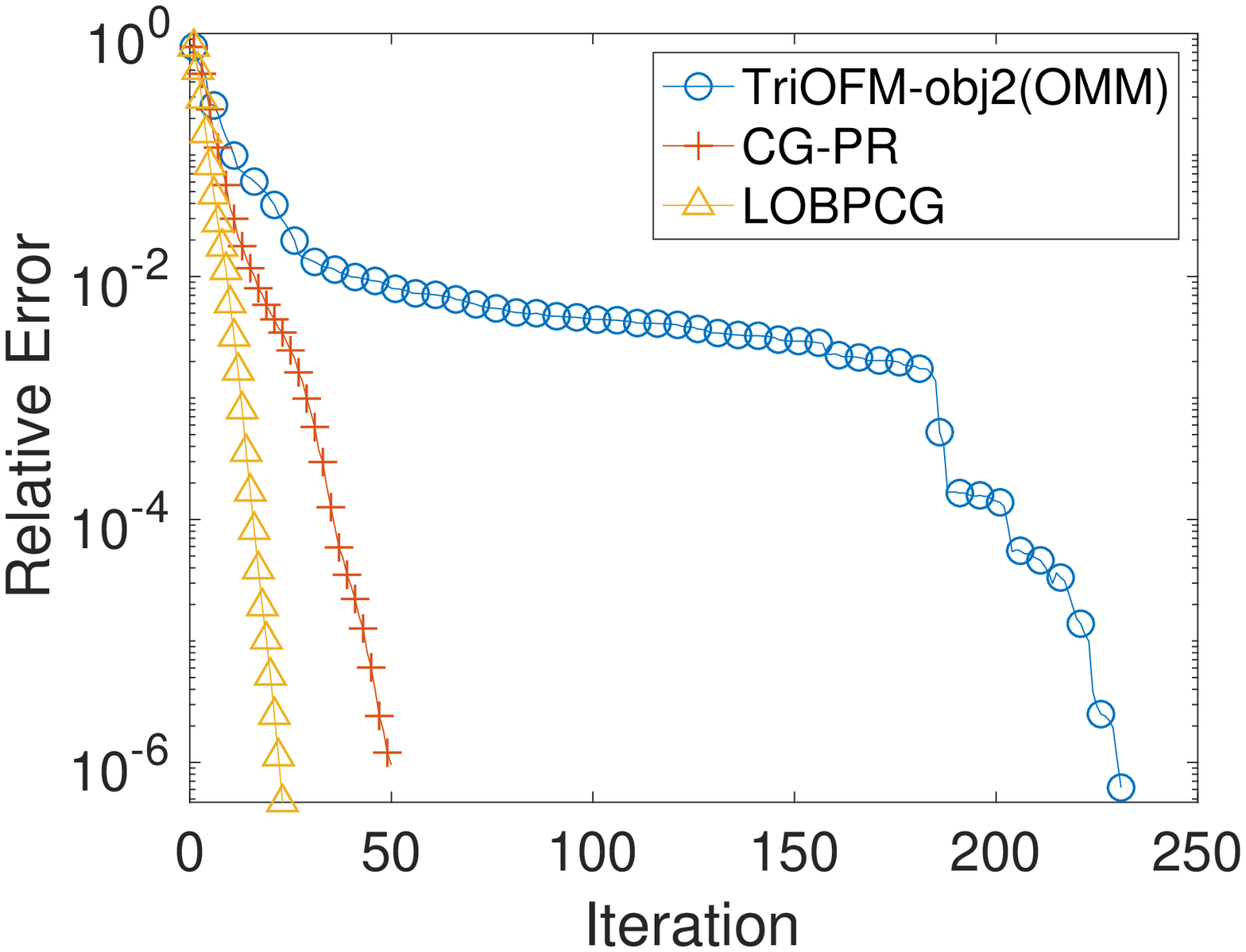}
}
\subfigure[Relative error vs CPU time]{
\includegraphics[width=0.46\textwidth]{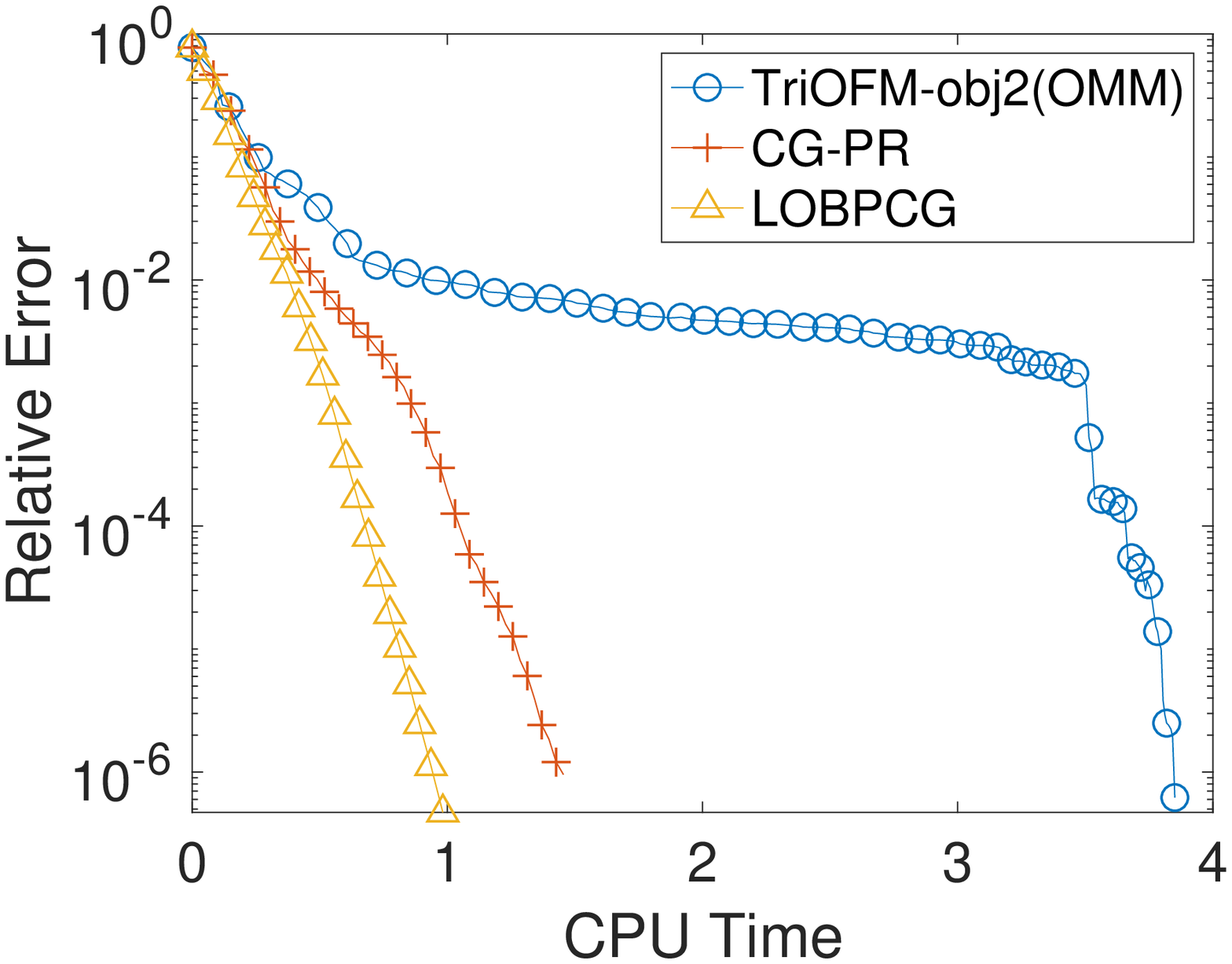}
}
\caption{Comparison for computing the top-10-eigenvalue problem of a $10^4$-by-$10^4$ matrix with randomly distributed eigenvalues.}
\label{fig:rand-dstribution}
\end{figure}
 
\begin{figure}[H]
\centering
\subfigure[Relative error vs iteration]{
\includegraphics[width=0.46\textwidth]{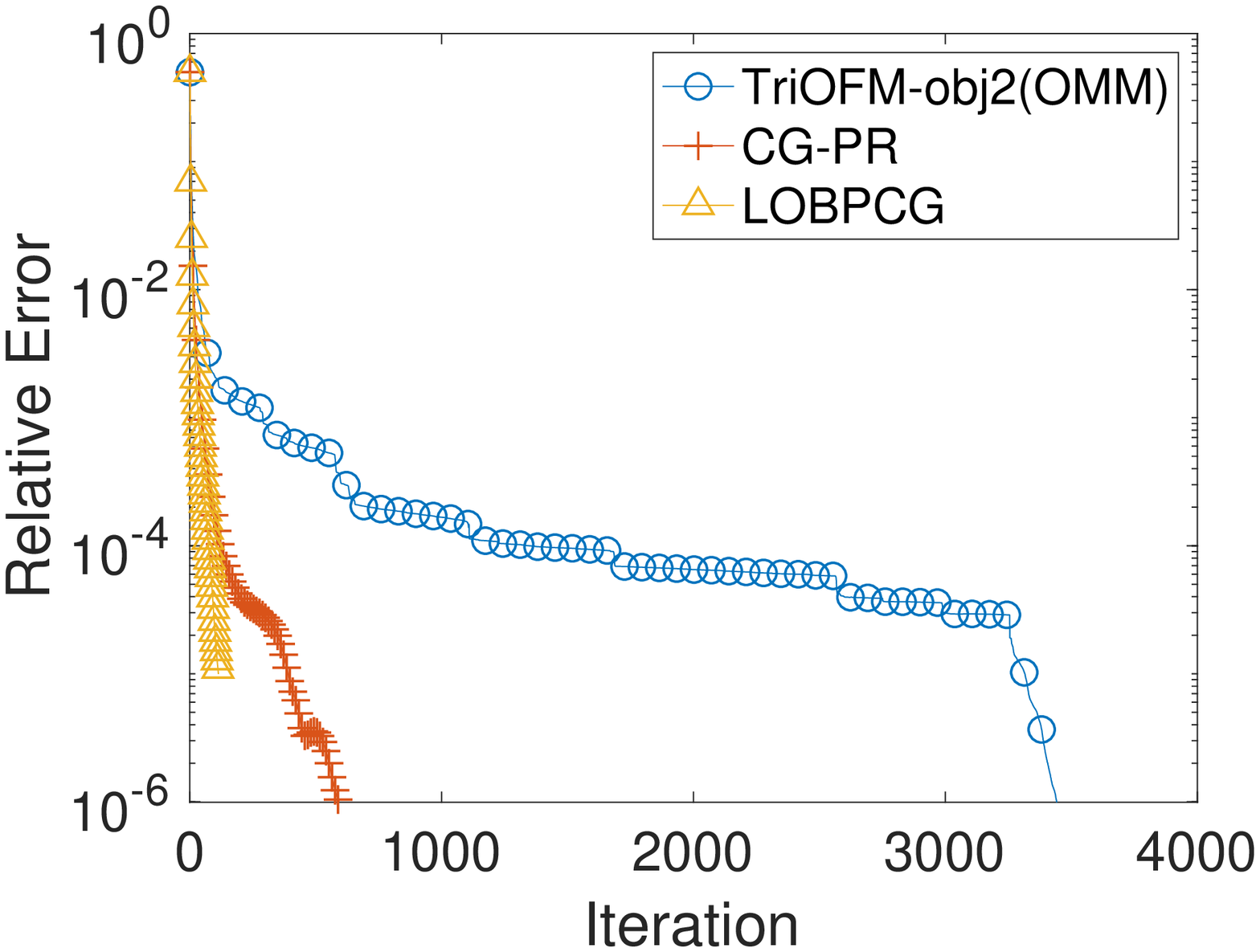}
}
\subfigure[Relative error vs CPU time]{
\includegraphics[width=0.46\textwidth]{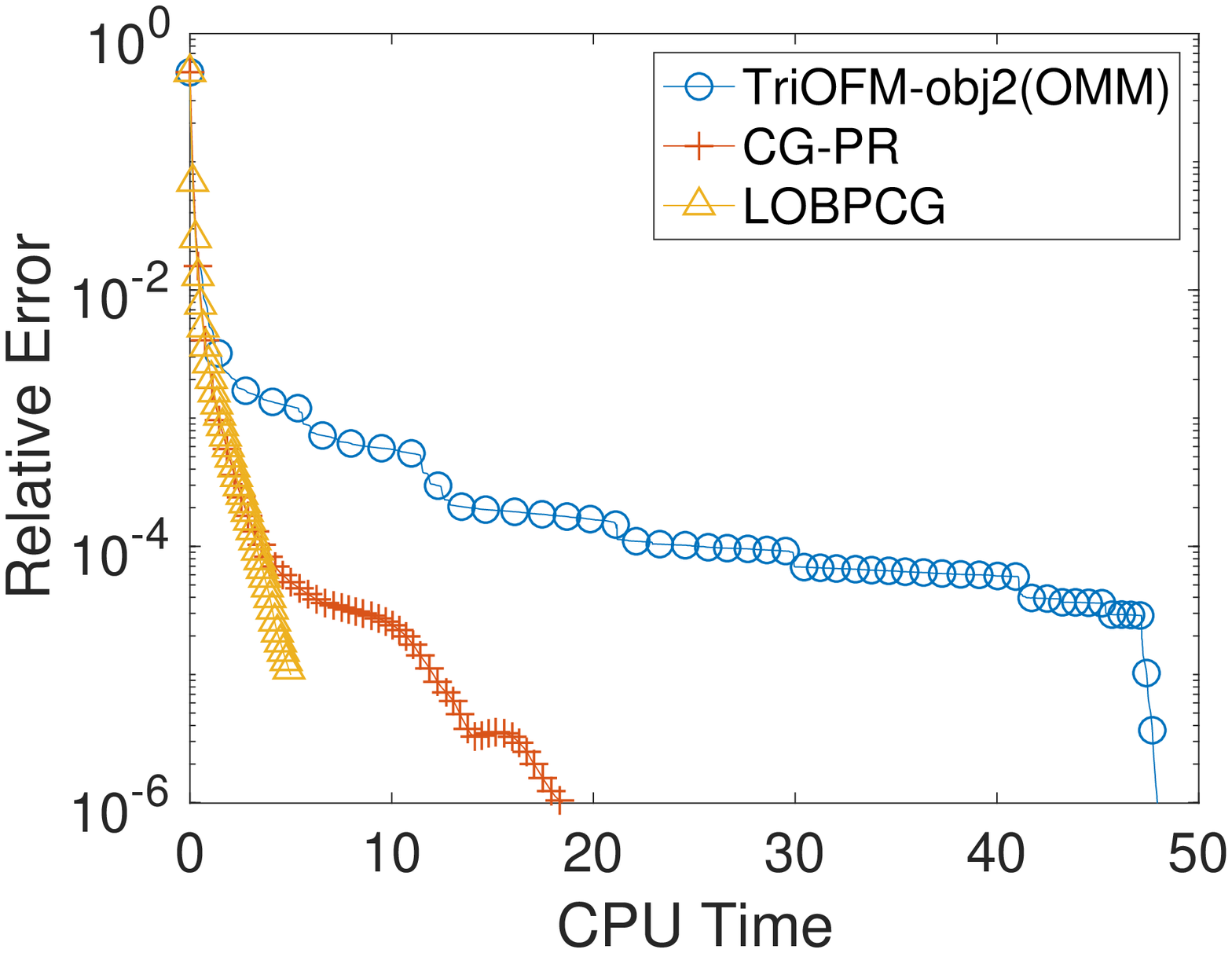}
}
\caption{Comparison for computing the  top-10-eigenvalue problem of a $10^4$-by-$10^4$ matrix with uniformly distributed eigenvalues.}
\label{fig:dct_uniform}
\end{figure}

\begin{figure}[H]
\centering
\subfigure[Relative error vs iteration]{
\includegraphics[width=0.46\textwidth]{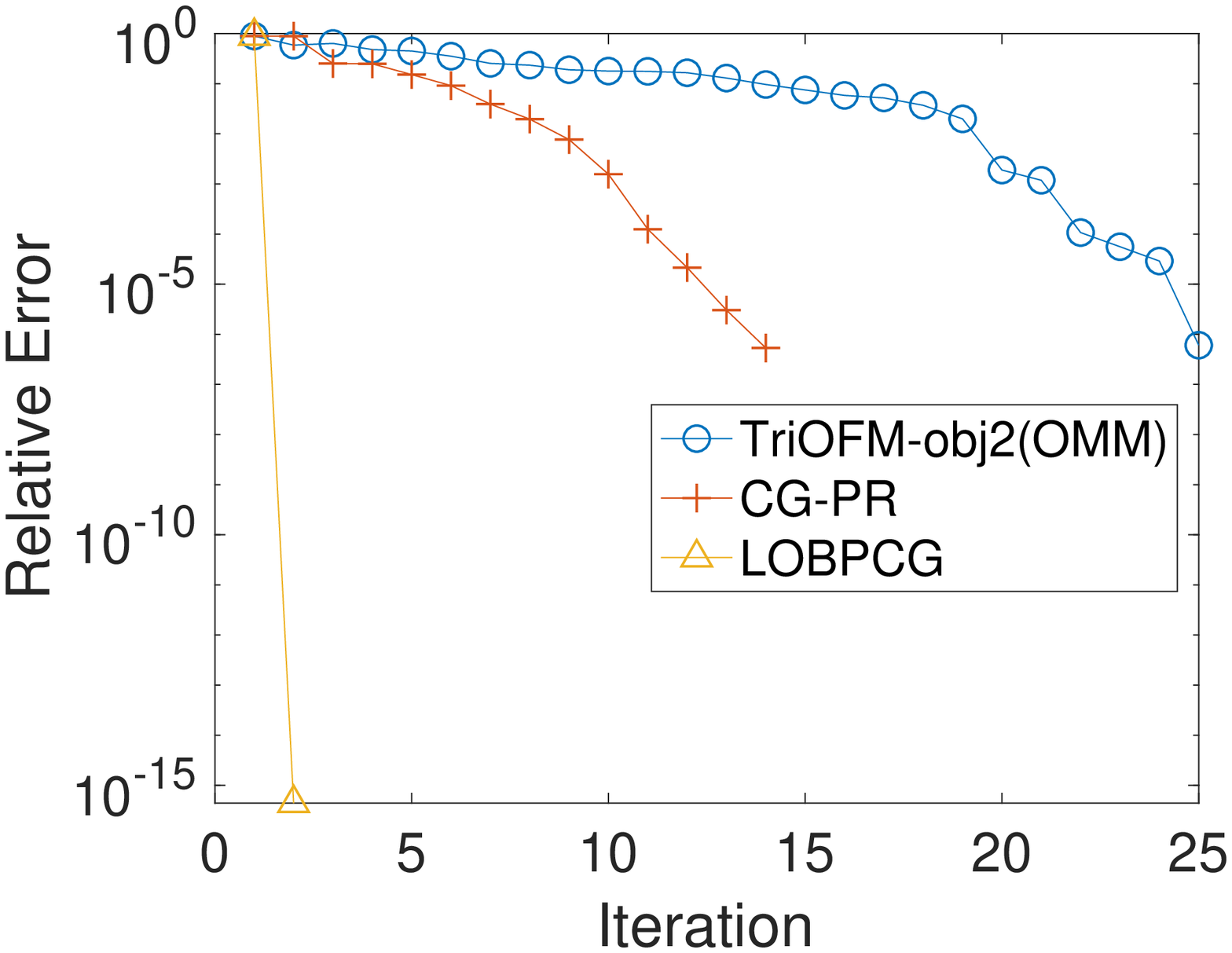}
}
\subfigure[Relative error vs CPU time]{
\includegraphics[width=0.46\textwidth]{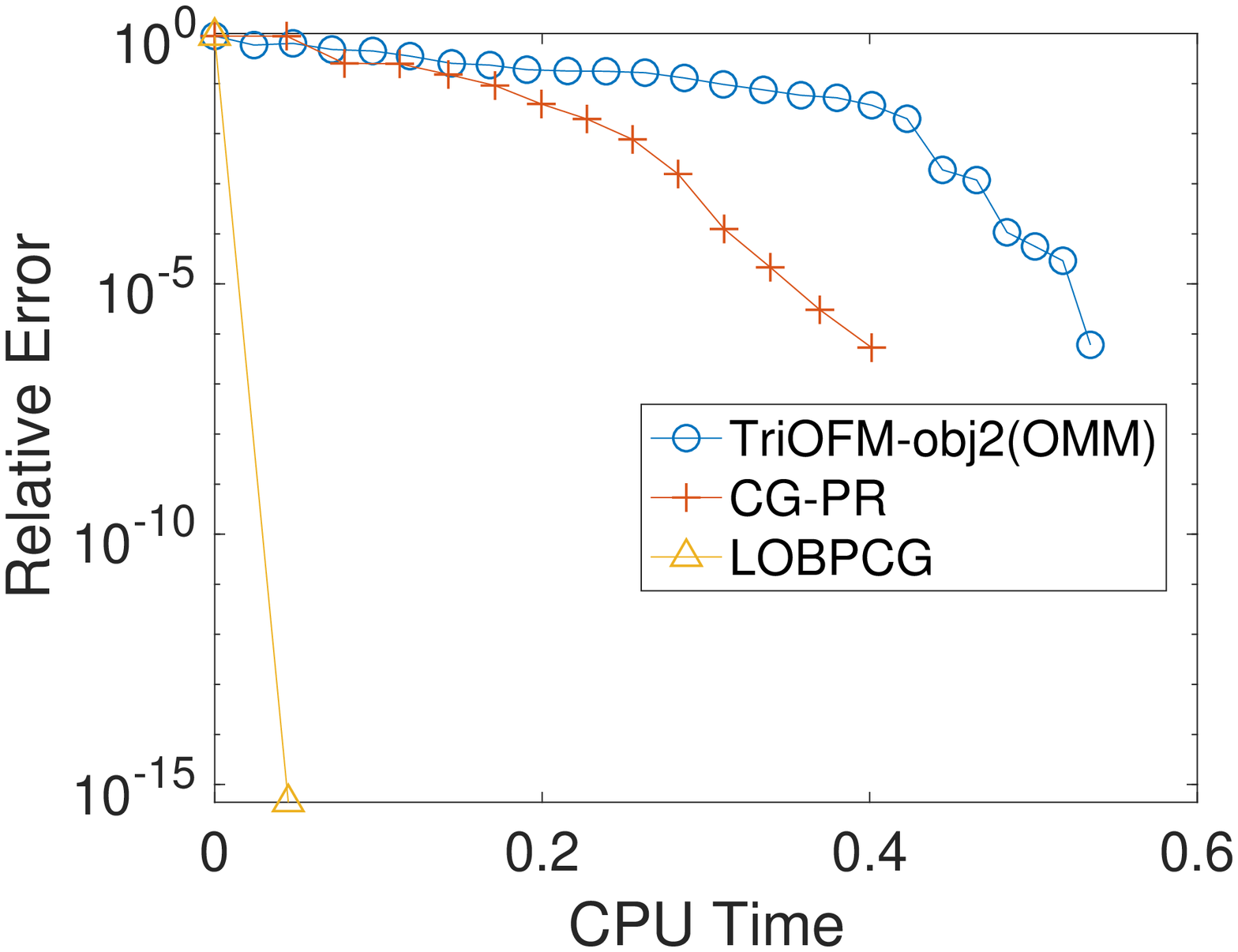}
}
\caption{Comparison for computing the  top-10-eigenvalue problem of a $10^4$-by-$10^4$ matrix with U-shape distributed eigenvalues.}
\label{fig:dct_ushape}
\end{figure}

\begin{figure}[H]
\centering
\subfigure[Relative error vs iteration]{
\includegraphics[width=0.46\textwidth]{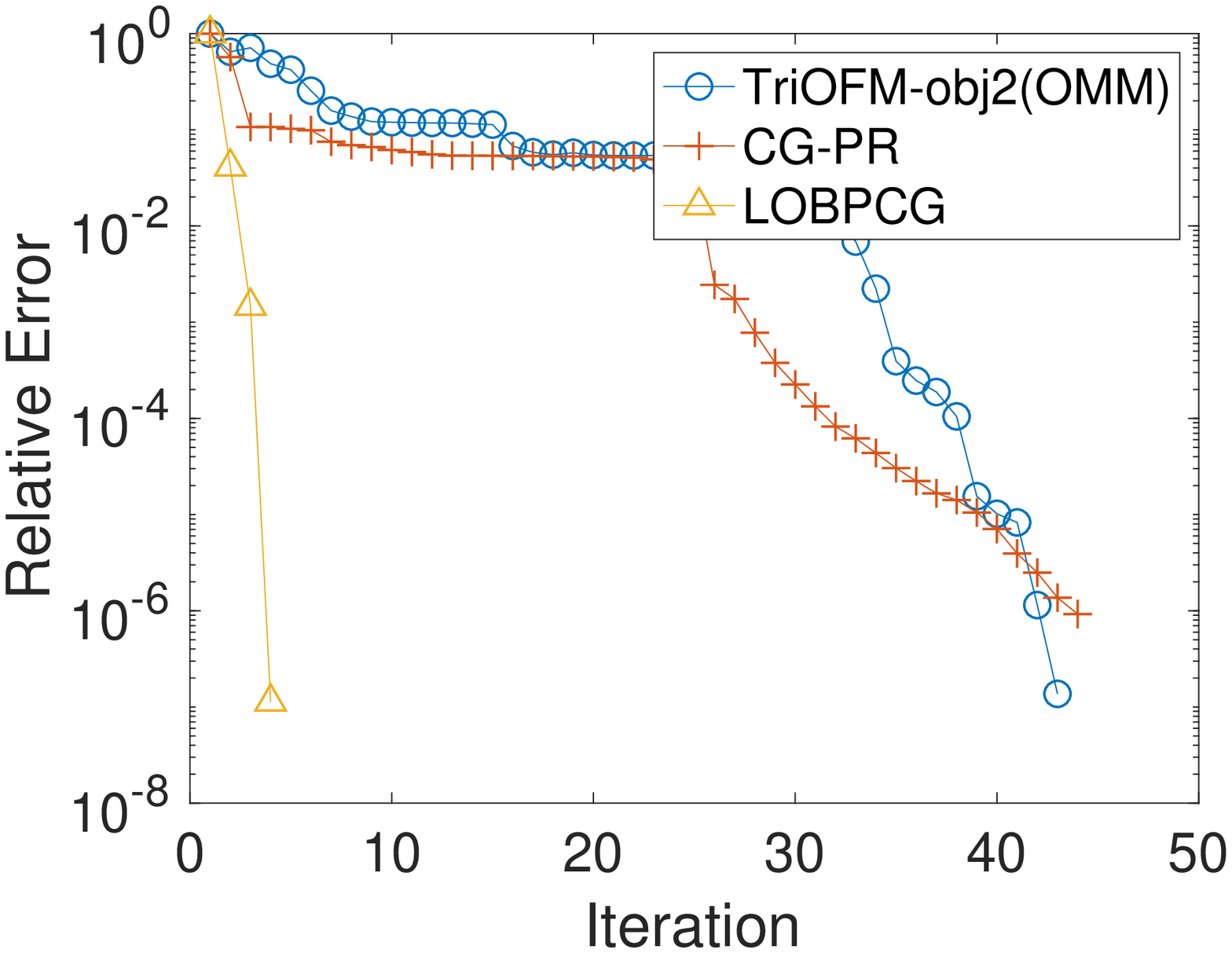}
}
\subfigure[Relative error vs CPU time]{
\includegraphics[width=0.46\textwidth]{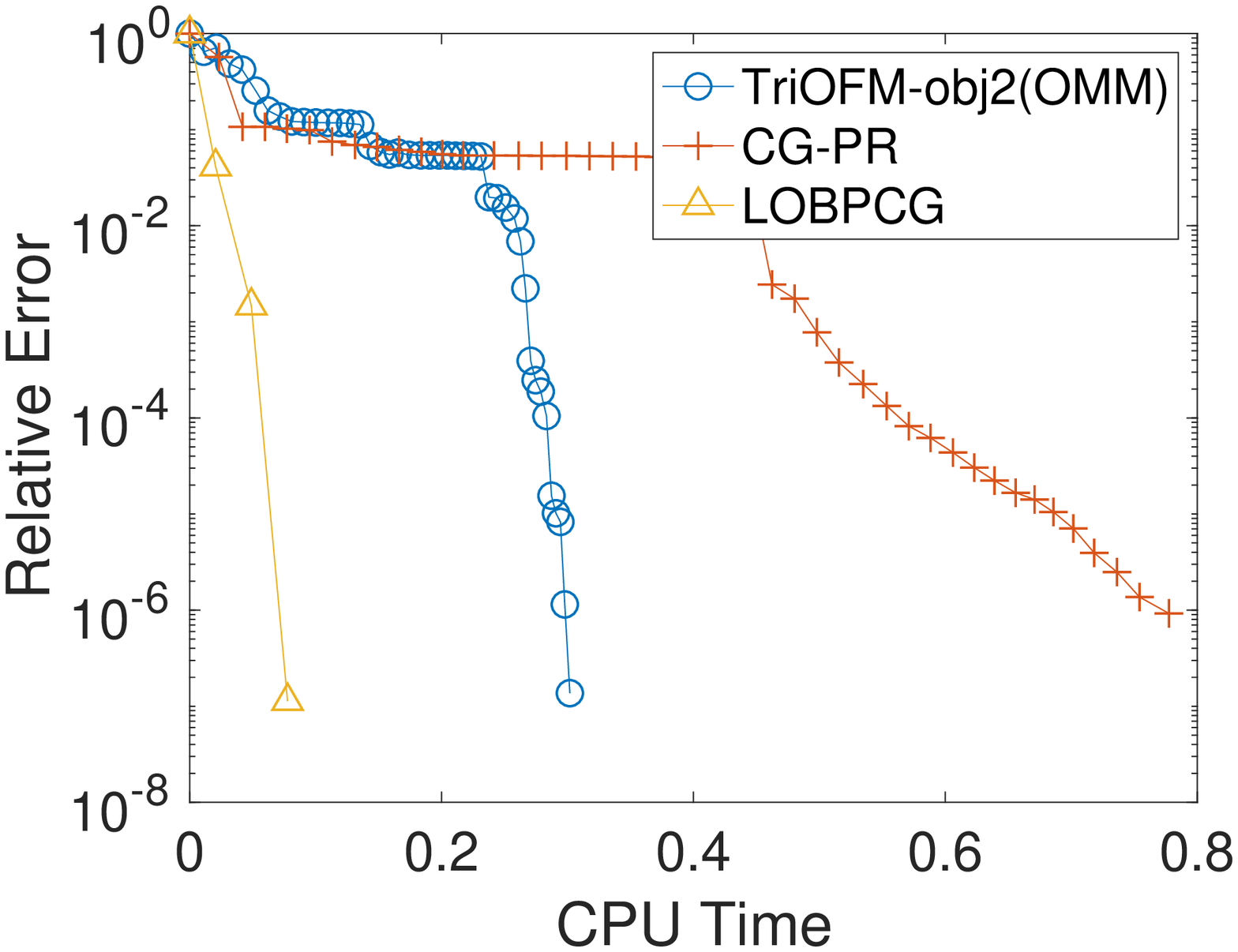}
}
\caption{Comparison for computing the top-5-eigenvalue problem of a $10^4$-by-$10^4$ matrix with logarithm distributed eigenvalues.}
\label{fig:log-2}
\end{figure}

\begin{figure}[H]
\centering
\subfigure[Relative error vs iteration]{
\includegraphics[width=0.46\textwidth]{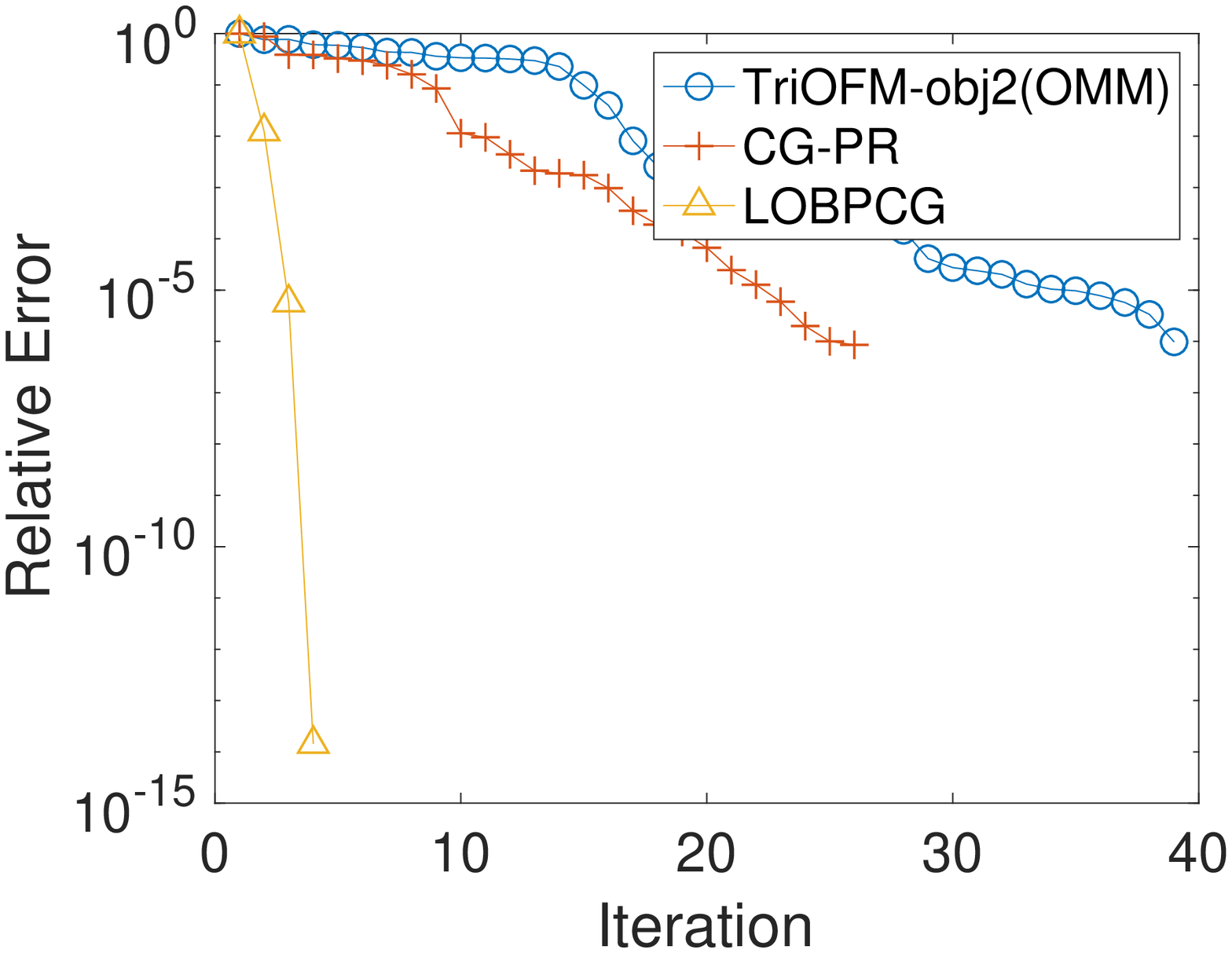}
}
\subfigure[Relative error vs CPU time]{
\includegraphics[width=0.46\textwidth]{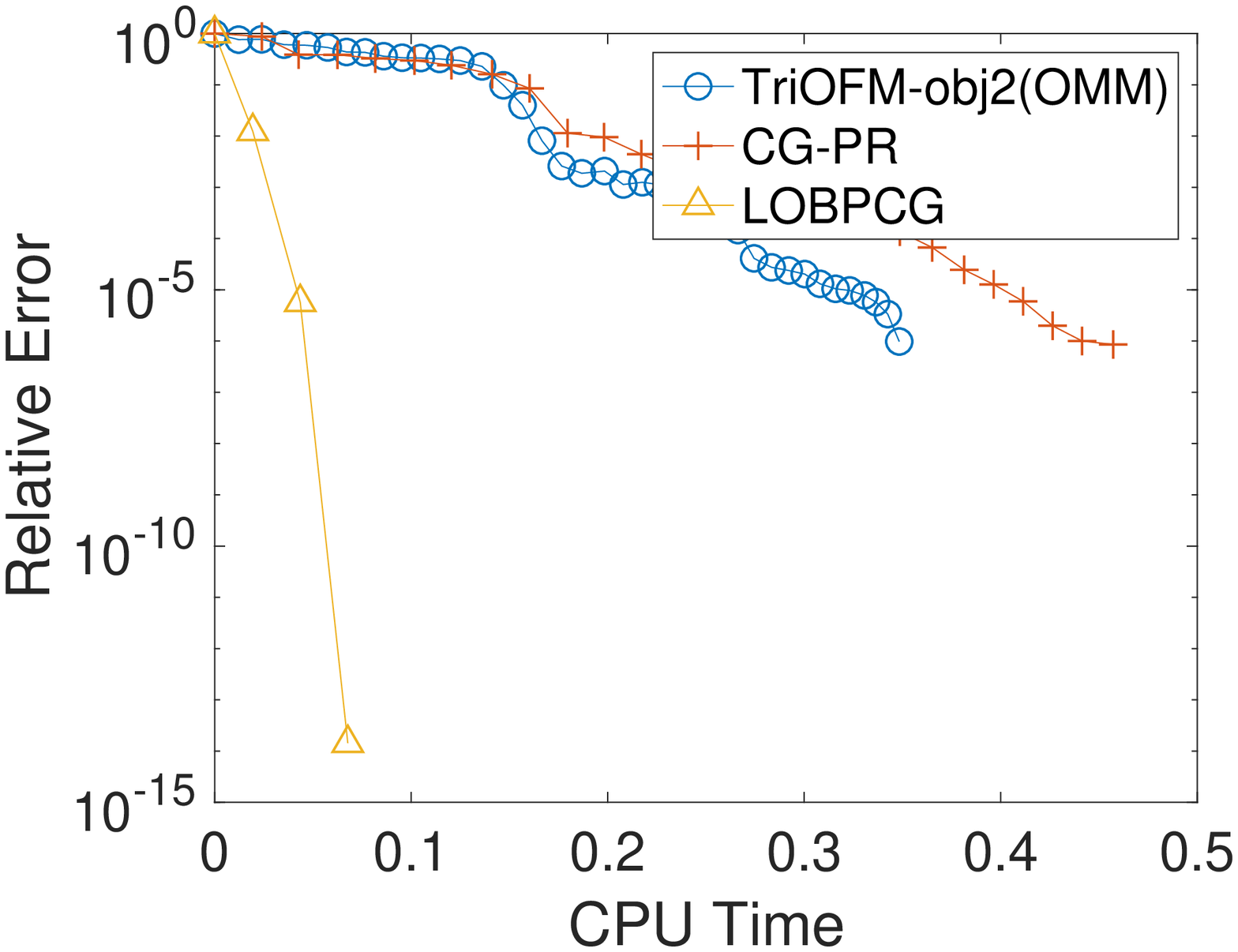}
}
\caption{Comparison for computing the top-5-eigenvalue problem of a $10^4$-by-$10^4$ matrix with eigenvalues $\lambda_1+C\geq \lambda_2+C\geq \dots \geq \lambda_5+C\geq \lambda_{5+1}\geq \cdots\geq \lambda_n$,
where $C=\lambda_1$ and $\lambda_1\geq \lambda_2\geq  \cdots\geq \lambda_n$ has a log distribution.}
\label{fig:log-3}
\end{figure}

\subsection{Hermitian PSD matrices}
It is shown in \cite{zheng2023riemannian} that Algorithm \ref{alg:RCG} can be used for finding the top eigenvalues of a Hermitian PSD matrix. We test Algorithm \ref{alg:RCG}  on \ref{min_prob_BM} for a matrix $A$ with eigenvectors defined by 2D Fast Fourier Transform. Namely, the linear operator of applying $A$ to a 2D array $u$ is defined by 
\[
Au = ifft2( \Sigma.* fft2(u)),
\]
where $.*$ denotes the entrywise product and $\Sigma $ is a 2D array consisting of nonnegative eigenvalues of $A$. 

The performance of the CG-PR method is shown in Figure \ref{fig:Hermitian-ex} for
four kinds of eigenvalue distributions in such a Hermitian PSD matrix.

\begin{figure}[H]
\centering
\subfigure[Relative error vs iteration]{
\includegraphics[width=0.46\textwidth]{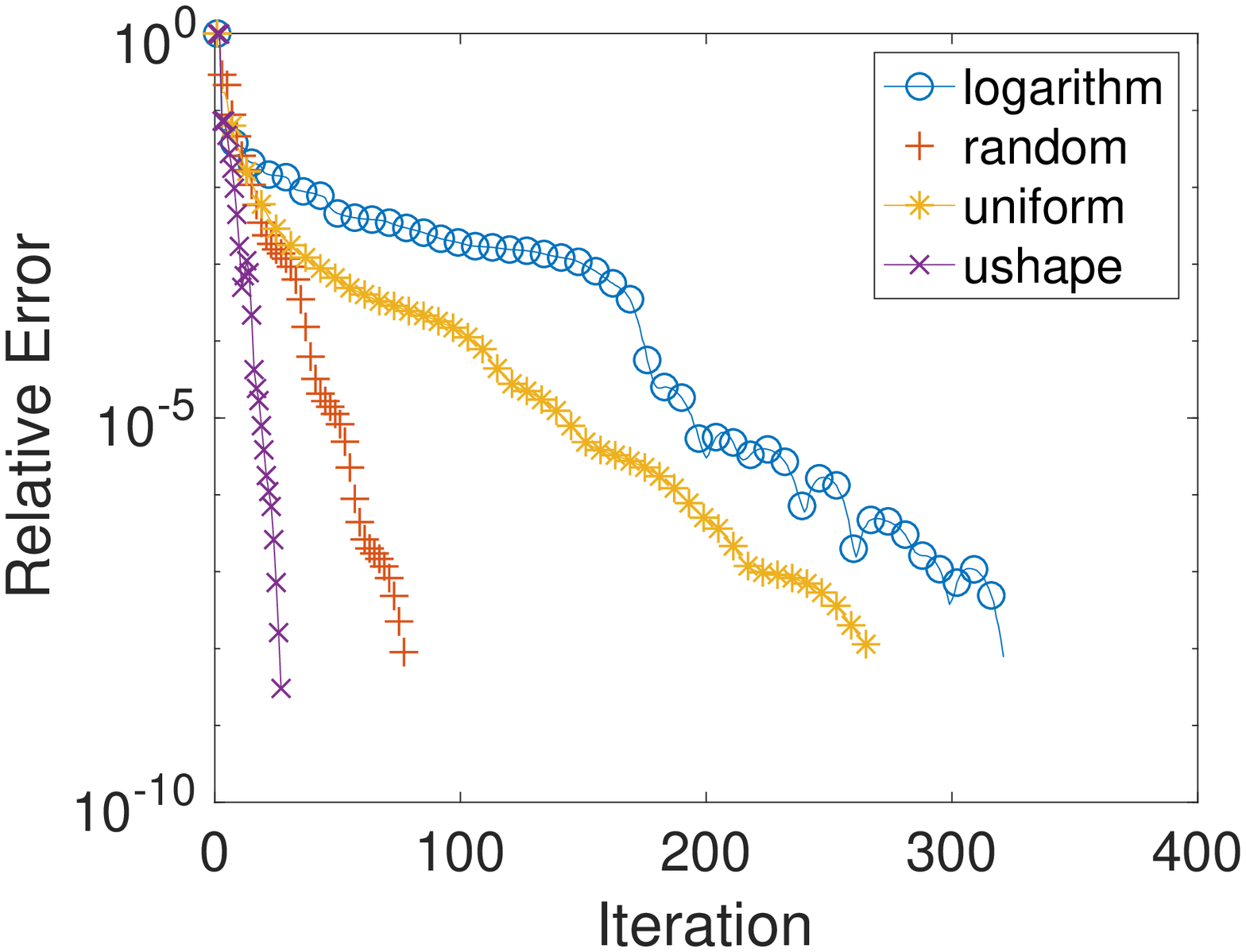}
}
\subfigure[Relative error vs CPU time]{
\includegraphics[width=0.46\textwidth]{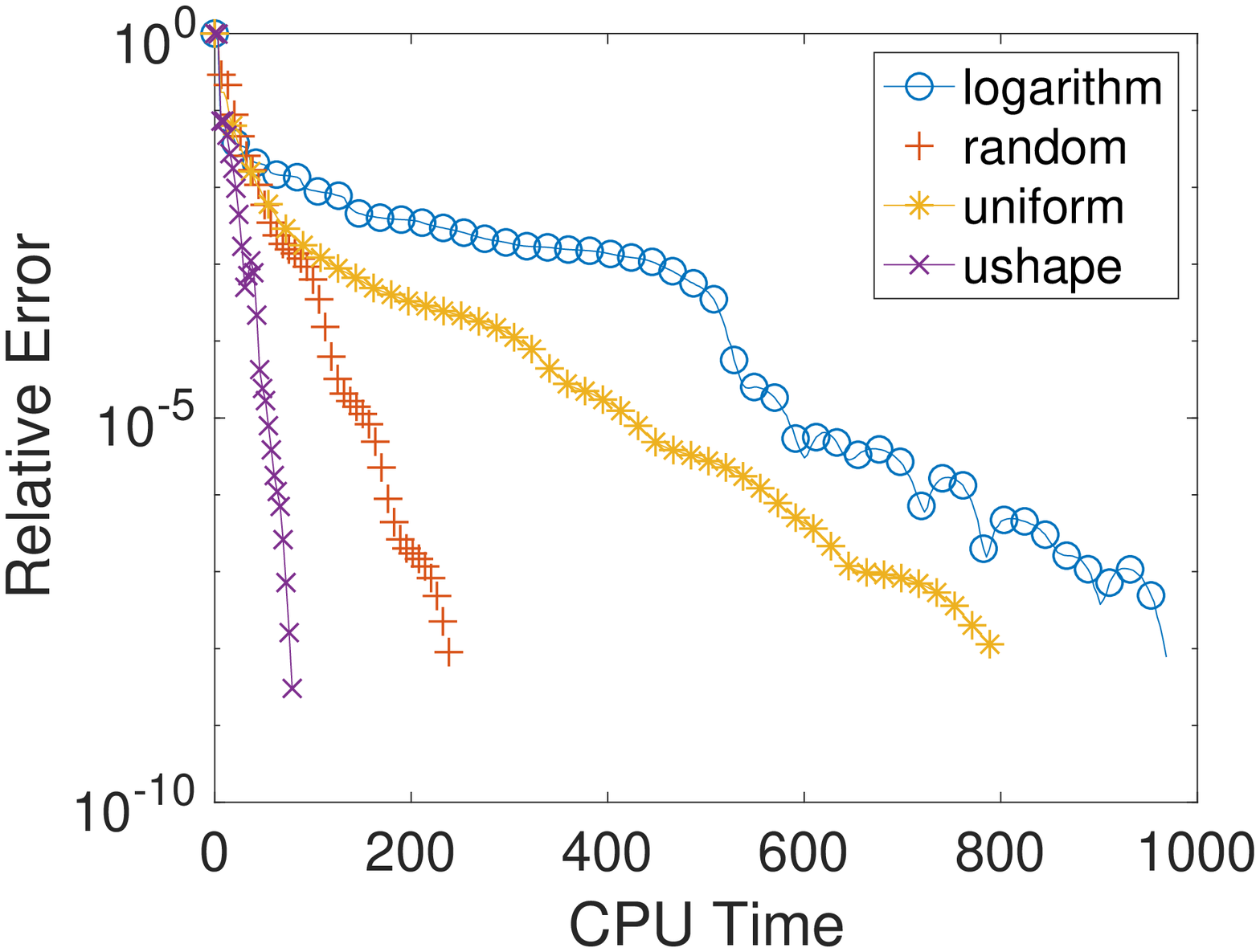}
}
\caption{The CG-PR method for the top-10-eigenvalue problem with rank-1000 Hermitian matrices of $10^6$-by-$10^6$ with different distributions of eigenvalues.}
\label{fig:Hermitian-ex}
\end{figure}

\subsection{Smallest eigenvalues}
\subsubsection{Inverse 2D Laplacian matrix}
One technique to find the smallest eigenvalues of a given invertible matrix $A$ is through the shift-and-inverse method. That is,  to find the largest eigenvalues of 
$(A+\mu I)^{-1}$, where $\mu >0 $ is a shift constant such that $A+\mu I$ becomes positive definite. We use this method to find the smallest eigenvalues of the 2D Laplacian matrix $A$ as in $\eqref{eqn:2d_laplacian_matrix}$. 

Notice that the top eigenvalues of $A^{-1}$  almost follow a logarithm distribution. Based on our observation, we can choose $\mu$ appropriately to make the top eigenvalues of $(A+\mu I)^{-1}$ have a uniform distribution to accelerate the convergence of the CG method. Since we know the true eigenvalues of $A$, we shift it by choosing $\mu$ to be the smallest desired eigenvalue. That is, suppose the smallest $r$ eigenvalues of $A$ is $\sigma_1 \leq \sigma_2 \leq \dots \leq \sigma_r$.  Then we choose $\mu = \sigma_1$. As a result the top eigenvalues of $(A+\mu I)^{-1}$ would be $\frac{1}{\sigma_1+\sigma_1} \geq \frac{1}{\sigma_2 + \sigma_1} \geq \dots \geq \frac{1}{\sigma_r + \sigma_1}$ that almost follows a uniform distribution. A fast matrix inversion is implemented by using the eigendecomposition of the matrix. The performance is shown in Figure \ref{fig:small-eig-1}
and Figure \ref{fig:small-eig-2}.

\begin{figure}[H]
\centering
\subfigure[Relative error vs iteration]{
\includegraphics[width=0.46\textwidth]{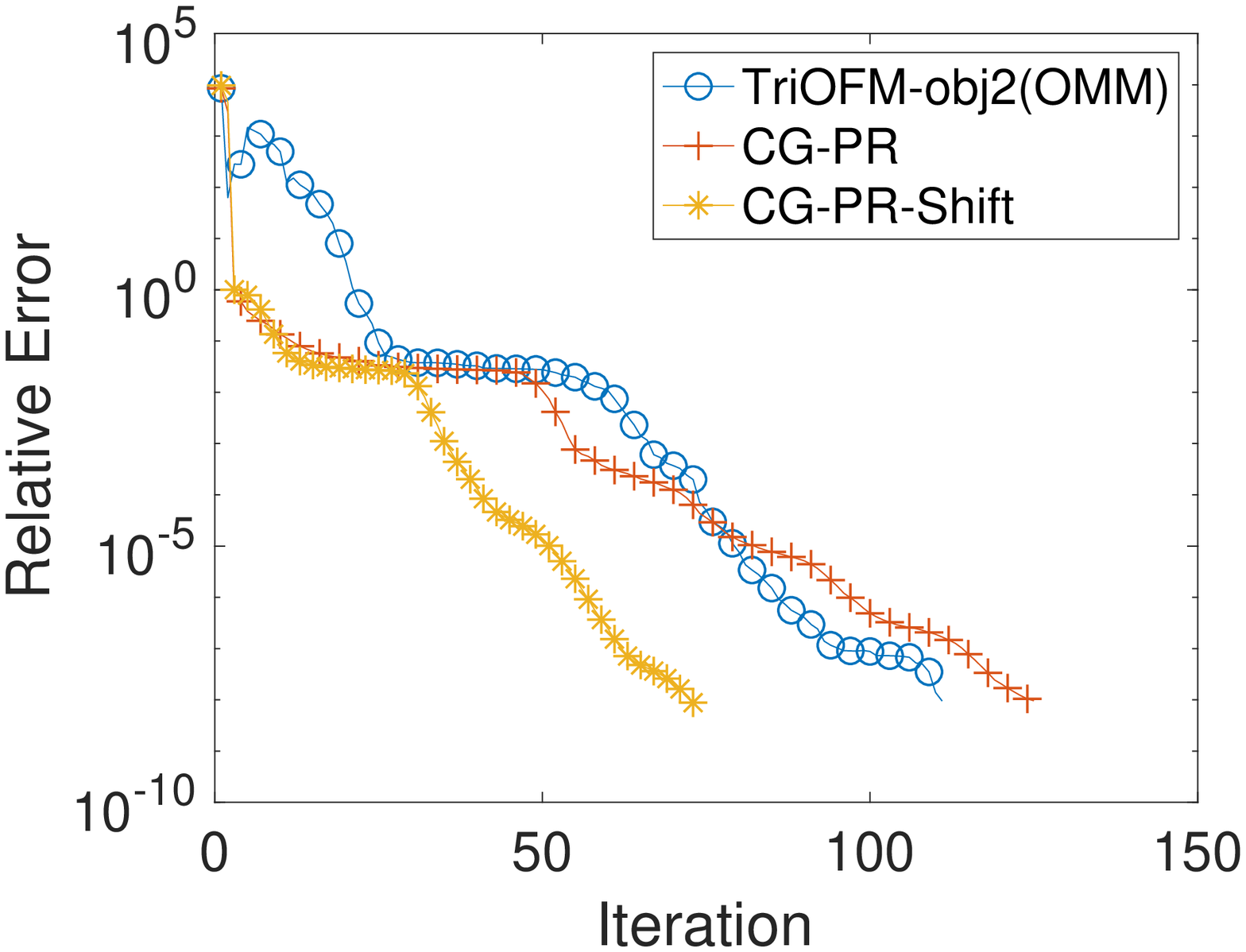}
}
\subfigure[Relative error vs CPU time]{
\includegraphics[width=0.46\textwidth]{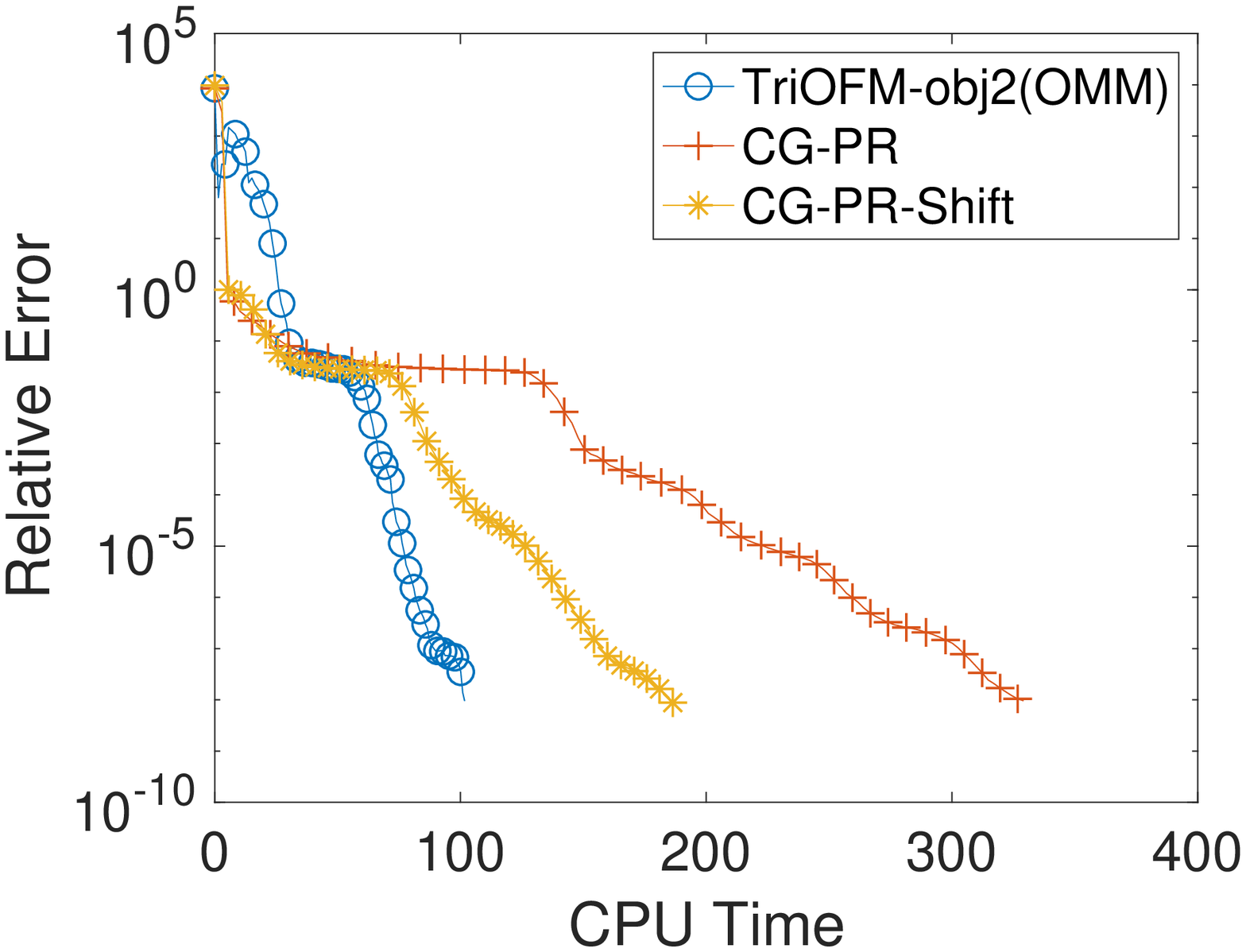}
}
\caption{The shift-and-inverse method on the smallest-10-eigenvalue problem of a $10^6$-by-$10^6$ 2D-Laplacian matrix.}
\label{fig:small-eig-1}
\end{figure}

\begin{figure}[H]
\centering
\subfigure[Relative error vs iteration]{
\includegraphics[width=0.46\textwidth]{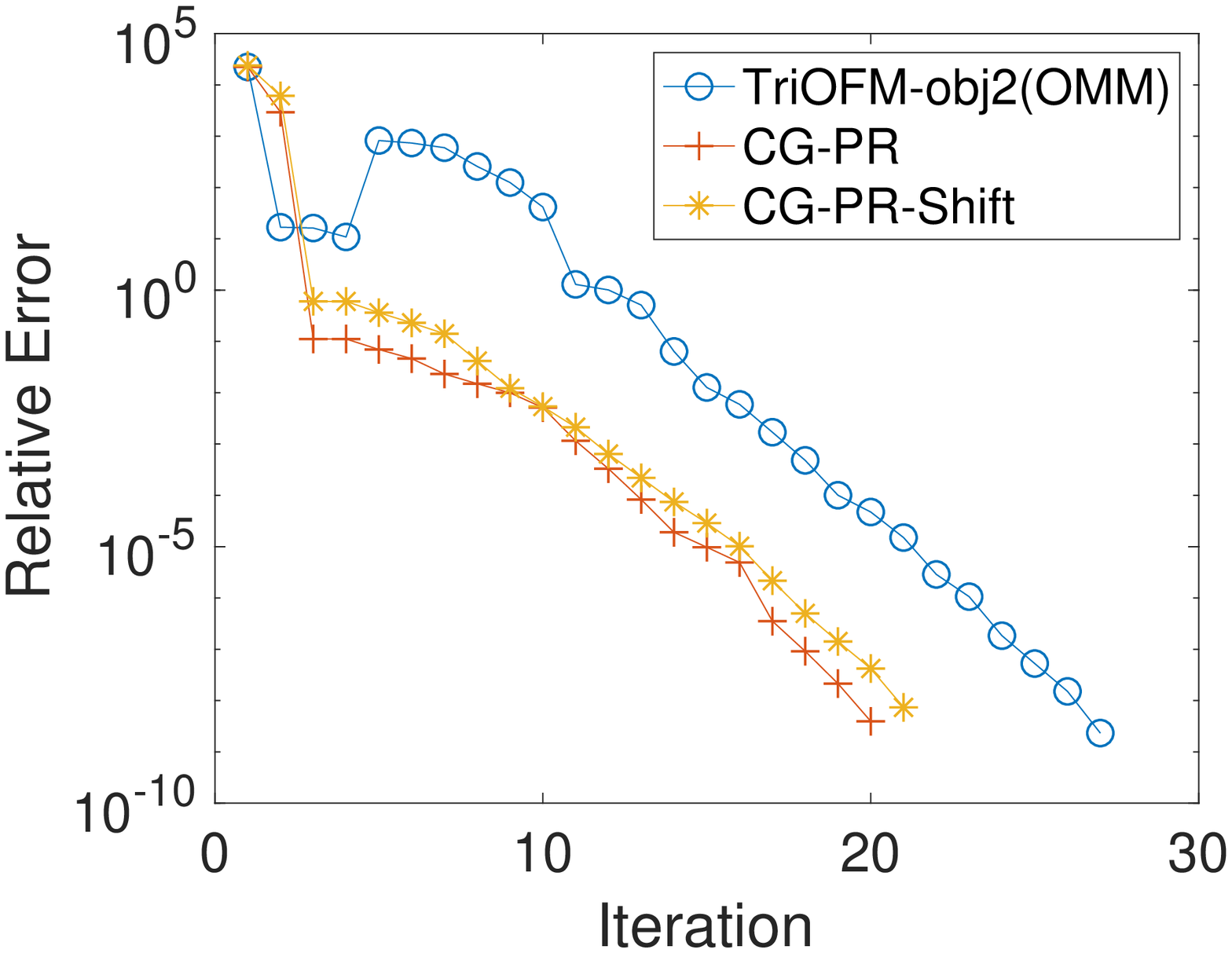}
}
\subfigure[Relative error vs CPU time]{
\includegraphics[width=0.46\textwidth]{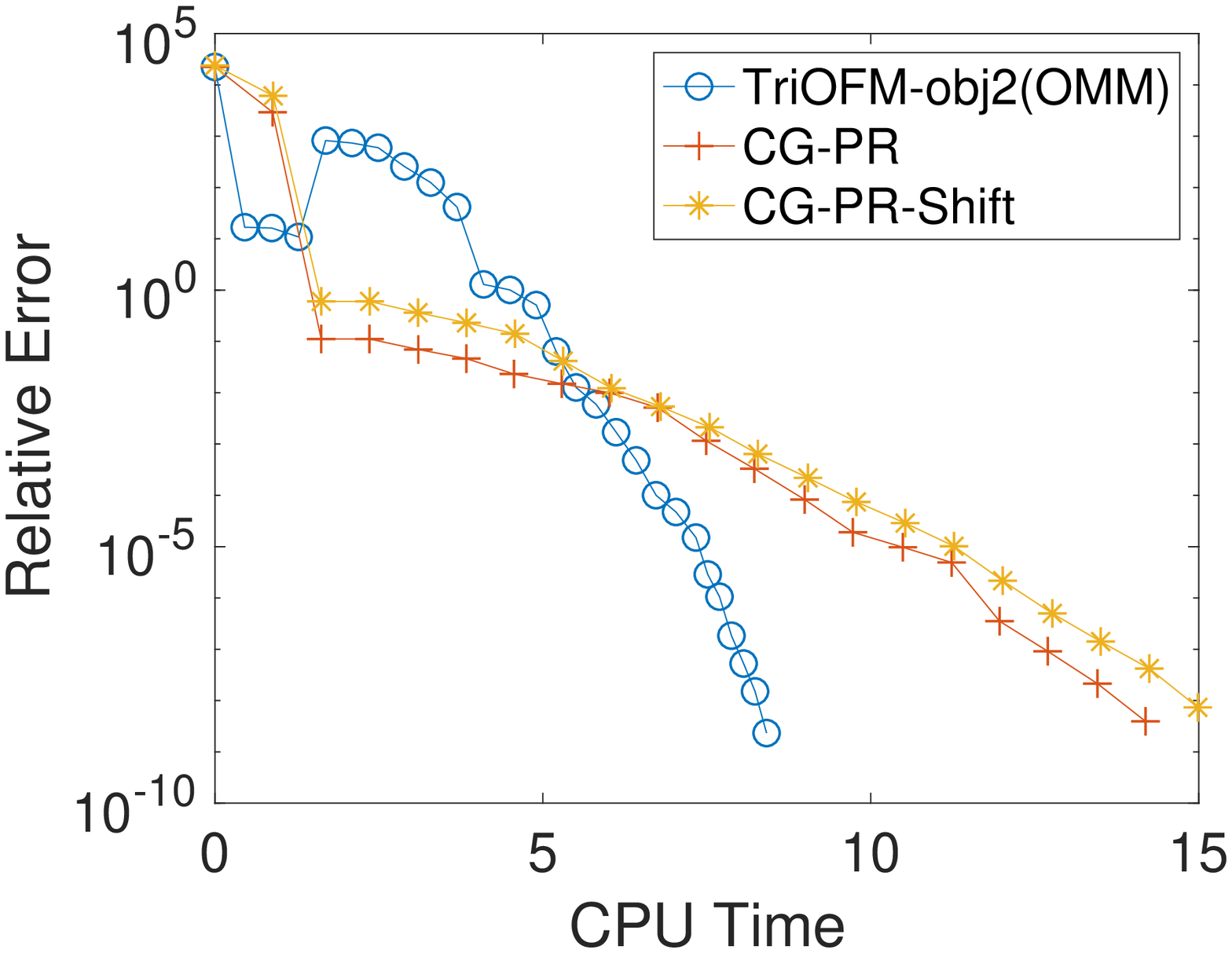}
}
\caption{The shift-and-inverse method on the smallest-3-eigenvalue problem of a $10^6$-by-$10^6$ 2D-Laplacian matrix.}
\label{fig:small-eig-2}
\end{figure}

\subsubsection{Negative 2D Laplacian matrix}
Another way to find the smallest eigenvalues of a given matrix $A$ is through the negative-shift method. That is, to consider finding the largest eigenvalues of $\mu I - A$, where $\mu >0$ is a shift constant such that $\mu I - A$ is positive semi-definite. We use this method to find the smallest eigenvalues of the 2D Laplacian matrix defined in $\eqref{eqn:2d_laplacian_matrix}$. 

Notice we need to shift at least the largest eigenvalue of $A$ to ensure that $\mu I -A$ is PSD. And once we find the top eigenvalues of $\mu I  - A$ we need to shift back and extract the smallest eigenvalues of $A$ by computing $\mu - (\mu - \sigma)$, where $\sigma$'s are the smallest eigenvalues of $A$. Hence when the condition number of $A$ is bad, i.e., if $\mu >> \sigma$, then we might lose a significant number of digits of accuracy for computing $\mu - (\mu - \sigma)$.  
In our numerical tests, we did not encounter this numerical accuracy issue.
The performance is shown in Figure \ref{fig:negative-shift}. Notice that the negative-shift method is much slower than 
the shift-and-inverse method, because of the different distributions of the largest eigenvalues of $\mu I-A$ and $(A+\mu I)^{-1}$.

\begin{figure}[h]
\centering
\subfigure[Relative error vs iteration]{
\includegraphics[width=0.46\textwidth]{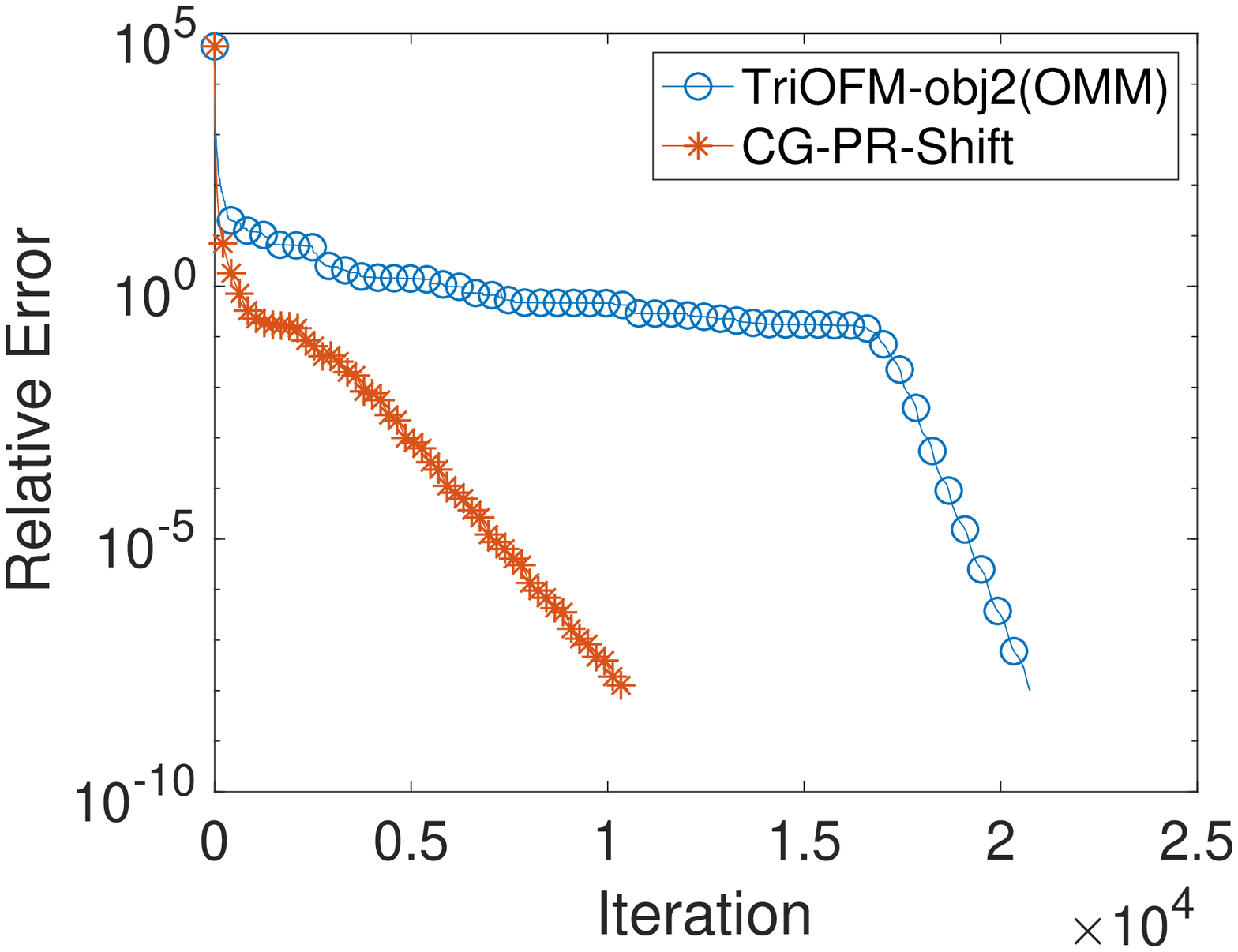}
}
\subfigure[Relative error vs CPU time]{
\includegraphics[width=0.46\textwidth]{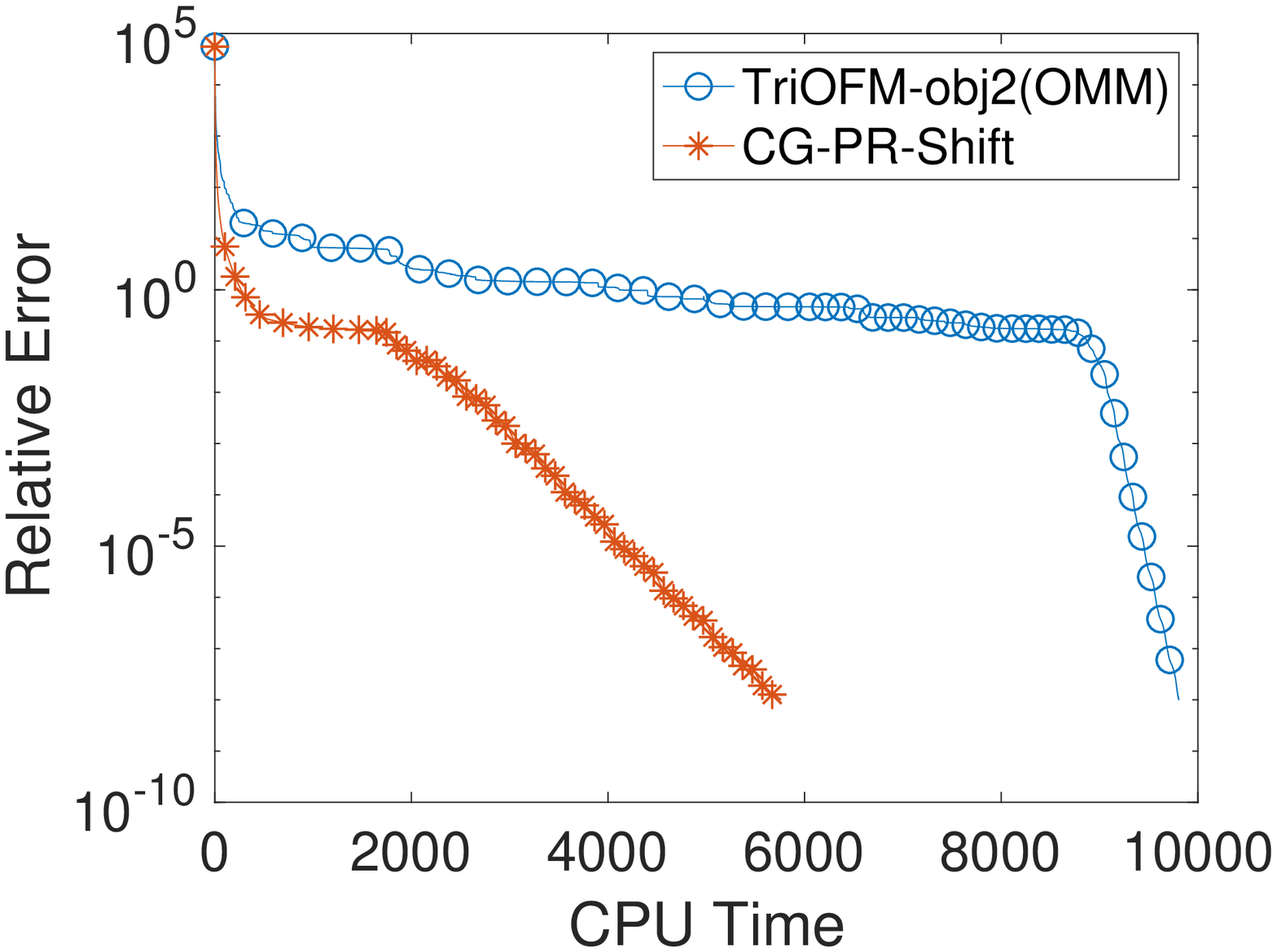}
}
\caption{The negative-shift method on the smallest-10-eigenvalue problem of a $10^6$-by-$10^6$ 2D-Laplacian matrix.}
\label{fig:negative-shift}
\end{figure}

\textcolor{black}{
\subsection{Negative 3D Laplacian matrix}
We repeat the same test as in previous subsection for a larger problem of finding the smallest eigenvalues of a 3D discrete Laplacian on a $500^3$ grid, which corresponds to a matrix of size 1.25E8$\times$1.25E8. We implement both the simple CG method \eqref{BMCG} and TriOFM method on a   Nvidia GPU A100  80G. 
\begin{figure}[H]
\centering
\subfigure[Relative error vs iteration]{
\includegraphics[width=0.46\textwidth]{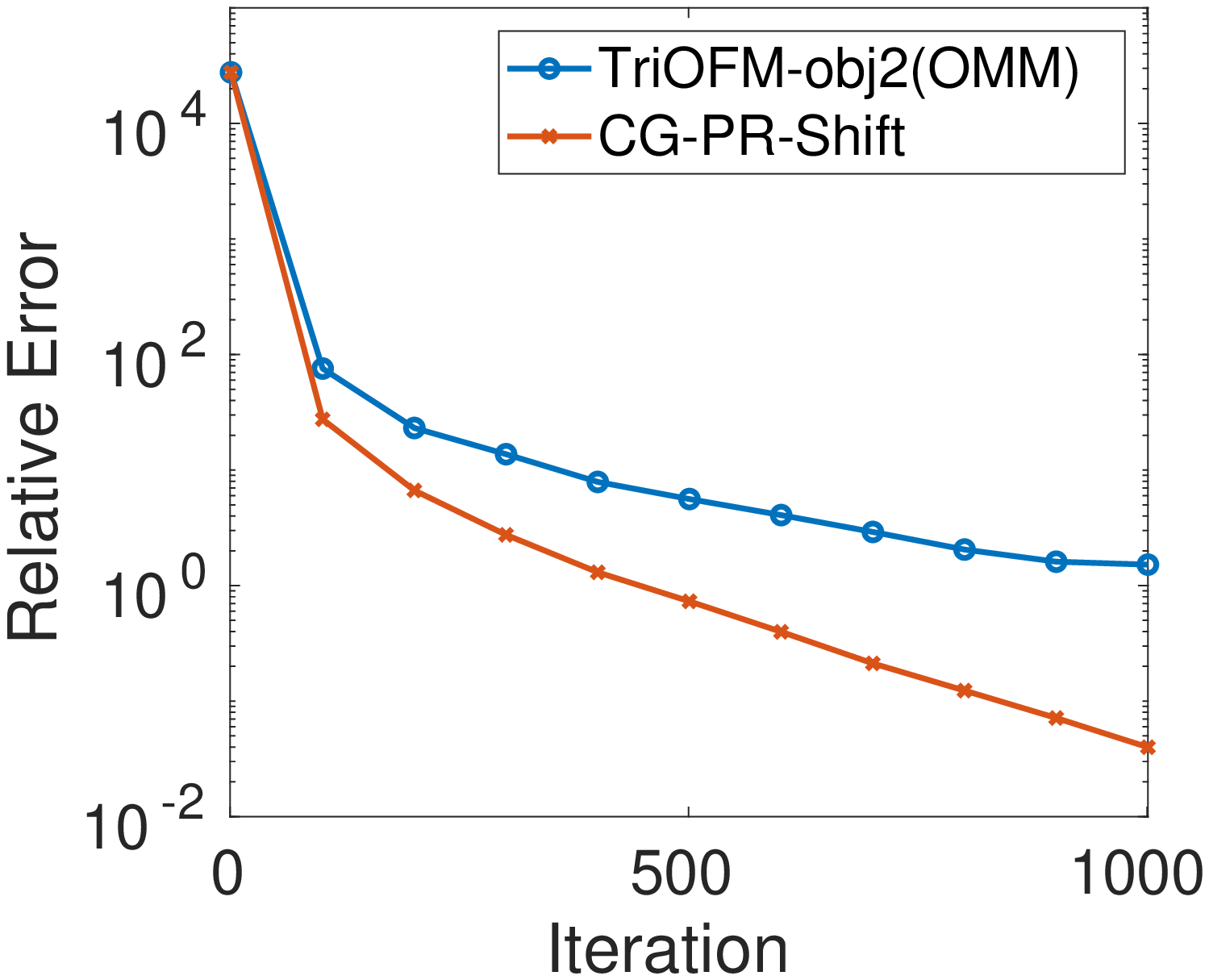}
}
\subfigure[Relative error vs GPU time]{
\includegraphics[width=0.46\textwidth]{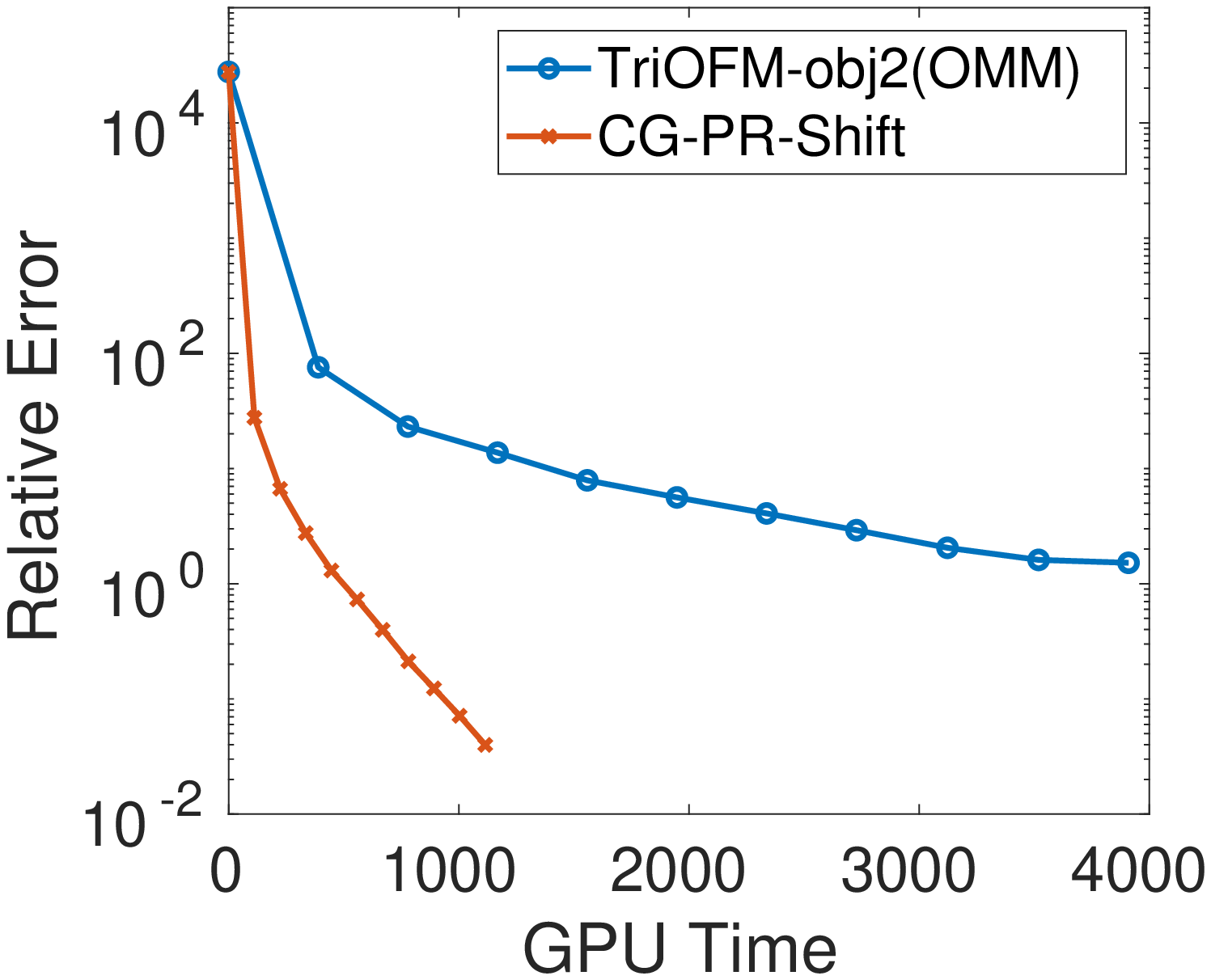}
}
\caption{\textcolor{black}{The shift-and-inverse method on the smallest-3-eigenvalue problem of a 3D-Laplacian matrix on a $500^3$ grid. The matrix size is 1.25E8$\times$1.25E8. Computation was done on Nvidia GPU A100  80G.}}
\label{fig:3d-laplacian}
\end{figure}
}
\subsection{Coordinate Riemannian gradient descent}
We consider applying the coordinate Riemannian gradient descent method described in Section \ref{sec-CRGD} to a 1D Laplacian matrix of size $n$-by-$n$ given by$
A = \frac{1}{\Delta x^2} K,$
where $\Delta x = \frac{1}{n+1}$ and $K$ is the tridiagonal matrix defined in \eqref{eqn:laplacian_matrix}. This example is only for the demonstration purpose of the coordinate gradient descent method. Choosing this simple $A$ makes it easy for the compact implementation of the matrix-vector multiplication of $Au$. One can also apply this method to any sparse matrix $A$ as long as one has the compact implementation of $M_k(Au)$ in $O(N)$, where $N$ is a constant independent of the problem size $n$.

As we can see from Figure \ref{fig:CRGD}, the CPU time for running the first 3000 iterations is independent of problem size. This demonstrated the $O(1)$ computational complexity of the coordinate Riemannian gradient descent method for leading eigenpairs.

\begin{figure}[htbp]
\centering
\subfigure[CPU time of the first 3000 iterations vs problem size $n = 100*2^k$ for $k$ goes from 4 to 13. Each iteration cyclically updates $N = 1000$ columns. ]{
\includegraphics[width=0.45\textwidth]{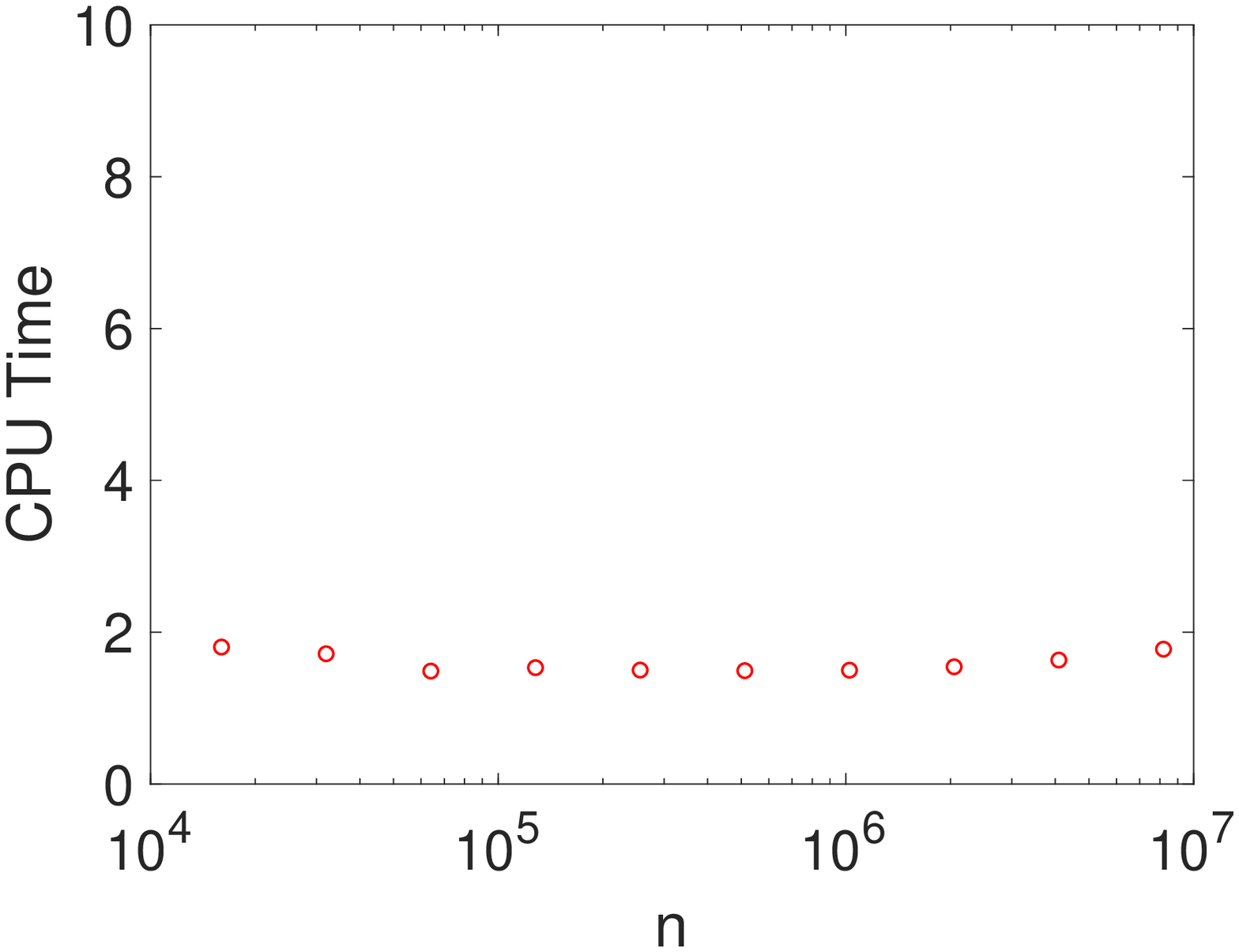}}
\subfigure[Relative error vs iteration. Problem size $n= 100*2^9$. Each iteration cyclically updates  $N = 100$ columns with constant step size $10^{-10}$.]{
\includegraphics[width=0.45\textwidth]{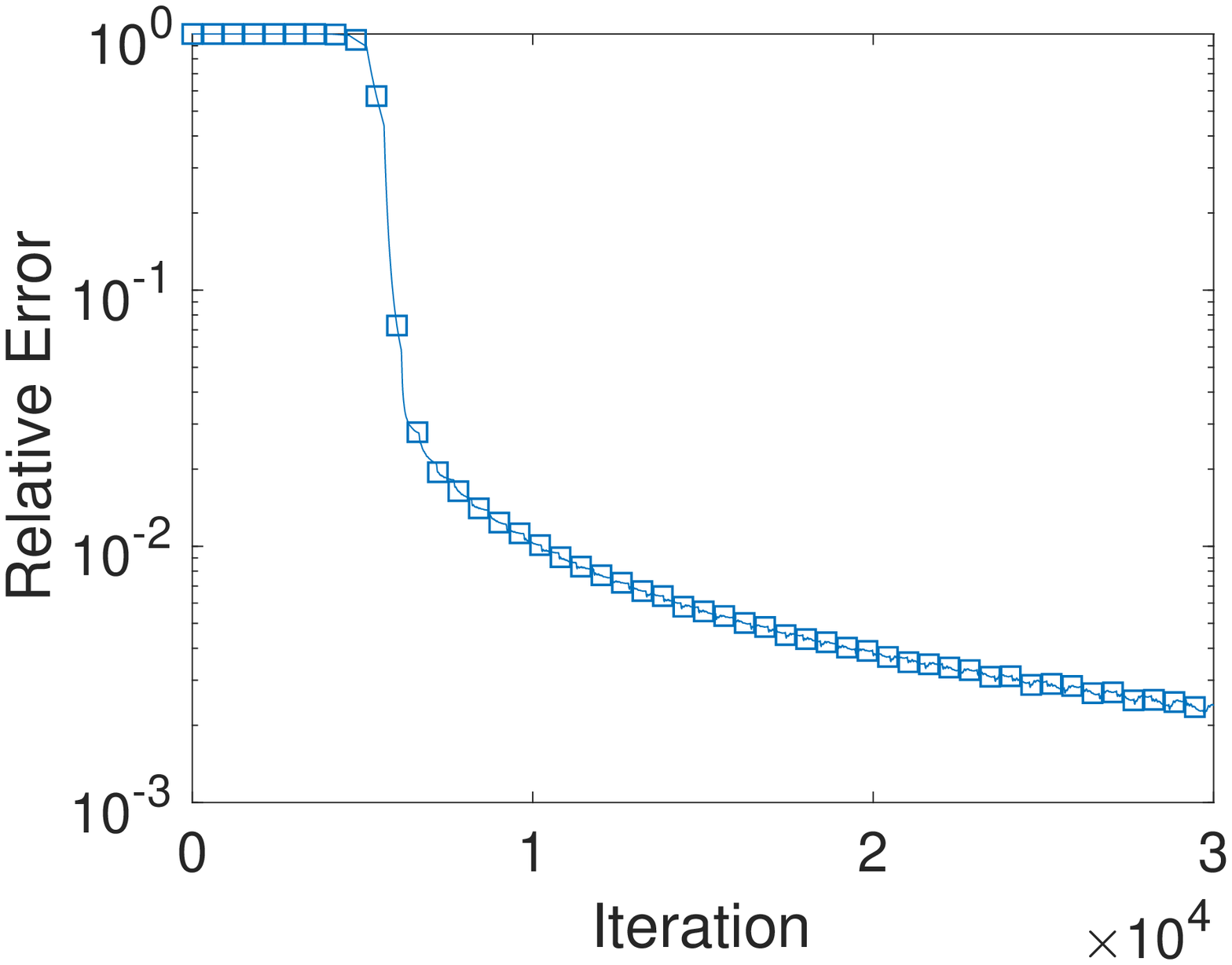}}
\caption{Coordinate Riemannian gradient descent for solving the top-10 eigenvalues of a Laplacian matrix.}
\label{fig:CRGD}
\end{figure}

\section{Conclusions}
\label{sec-remarks}

In this paper we have studied the orthogonalization-free method to find leading eigenpairs of a positive semi-definite Hermitian matrix via an unconstrained Burer-Monteiro formulation.  For this optimization problem, we have shown the equivalence between the nonlinear conjugate gradient method and a Riemannian conjugate gradient method on a quotient manifold with \textcolor{black}{the Bures-Wasserstein} metric, leading to a new understanding of the global convergence of the nonlinear conjugate gradient method in Burer-Monteiro formulation to a stationary point. \textcolor{black}{We have also shown that the simple coordinate descent method in Burer-Monteiro formulation is equivalent to a coordinate 
 Riemannian gradient descent method.}
\textcolor{black}{Numerical tests on large scale matrices have verified the numerical performance of the simple conjugate gradient method in Burer-Monteiro formulation for computing leading eigen-pairs, which is consistent with findings in the literatue. }

\section*{Acknowledgement}
S. Zheng and X. Zhang are supported by NSF DMS-2208518. H. Yang thanks Oracle Labs, part of Oracle America, Inc., for providing funding that supported research in the area of leading eigenvalue problems.
The authors are grateful to Yingzhou Li for providing the MATLAB code of TirOFM.





\bibliographystyle{elsarticle-num}
\bibliography{zotero, references}







\end{document}